\documentclass[12pt]{amsart}
\usepackage{amsmath}
\usepackage{amsfonts}
\usepackage{amssymb}
\usepackage{epsfig}
\newtheorem{theorem}{Theorem}[section]

\newtheorem{lm}[theorem]{Lemma}
\newtheorem{tr}[theorem]{Theorem}
\newtheorem{cor}[theorem]{Corollary}

\newtheorem{df}[theorem]{Definition}
\newtheorem{rem}[theorem]{Remark}
\newtheorem{pr}[theorem]{Proposition}
\usepackage{color}
\newcommand{\tc}[1]{\textcolor{red}{#1}}
\usepackage[notref,notcite]{showkeys}

\newcommand{\cal}{\mathcal}
\begin{document}
\title{Pseudocompact algebras and highest weight categories}
\author{Frantisek Marko}
\address{Penn State Hazleton, 76 University Drive, Hazleton PA 18202, USA}
\email{fxm13@psu.edu}
\author{Alexandr N. Zubkov}
\address{Omsk State Pedagogical University, Chair of Geometry, 644099 Omsk-99, Tuhachevskogo Embankment 14, Russia}
\email{zubkov@iitam.omsk.net.ru}
\begin{abstract} We develop a new approach to highest weight categories $\cal{C}$ with good (and cogood) posets of weights via pseudocompact algebras
by introducing ascending (and descending) quasi-hereditary pseudocompact algebras. For $\cal{C}$ admitting a Chevalley duality, we define and investigate tilting modules and Ringel duals of the corresponding pseudocompact algebras. Finally, we illustrate all these concepts on an explicit example of the general linear supergroup $GL(1|1)$.
\end{abstract}
\maketitle

\section*{Introduction}
The concept of the tilting module appeared in the early 1980's in the area of algebraic groups in \cite{don0}
and in finite-dimensional associative algebras in \cite{bb}, and it quickly became one of the most important concepts in the representation theory of algebras.

The highest weight categories were introduced in \cite{cps}. They established themselves as one of the cornerstone concepts that play significant roles in the areas of the representation theory of algebraic groups, Lie algebras and superalgebras, and associative algebras.
In many cases, like blocks of the category $\mathcal O$ and Schur algebras, every finitely-generated ideal $\Gamma$ of the partially ordered set of weights $\Lambda$ is finite. In these cases the full subcategory of objects belonging to $\Gamma$ is equivalent to a category of modules over a finite-dimensional quasi-hereditary algebra. Consequently, the concept of quasi-hereditary algebras was investigated intensively within the realm of Artin algebras. 

The concepts of tilting modules and quasi-hereditary algebras were synthesized in the concept of the Ringel dual that provides an important tool for relating quasi-hereditary algebras using endomorphisms of tilting modules.   

It was proved in \cite{z} that the category $GL(m|n)-smod$ of left rational supermodules over a general linear supergroup $GL(m|n)$ is the highest weight category. An analogous statement is valid for the category of right rational supermodules. However, the corresponding poset of weights $\Lambda$ does not satisfy the above condition that every finitely-generated ideal $\Gamma$ of the partially ordered set of weights $\Lambda$ is finite. Therefore, the established language and machinery of 
quasi-hereditary algebras and tilting modules are not applicable to this motivating example.

In this paper we suggest how to deal with some cases when the ideal $\Gamma$ of $\Lambda$ is infinite. 
Since a highest weight category is an abelian category of finite type, it can be regarded as a right comodule category over a coalgebra. Equivalently, this can be viewed as a left discrete module category over a pseudocompact algebra.  
If the corresponding set of weights $\Lambda$ (or a finitely-generated ideal $\Gamma$ of $\Lambda$) is not finite, then the corresponding pseudocompact algebra is necessary infinitely dimensional. 

In Section 1, we start by recalling definitions and basic results regarding highest weight categories, good filtrations, quasi-hereditary algebras, pseudocompact algebras and categories of comodules.

In Section 2, we first derive some results for pseudocompact algebras; then we restrict our attention to the case when $\Lambda$ satisfies mild assumptions, namely we assume that $\Lambda$ is good or cogood. (See Definitions \ref{d2} and \ref{d7}, respectively.) 
Corresponding to highest weight categories with finitely-generated good (or cogood) $\Lambda$, we then develop a theory of ascending (descending) quasi-hereditary pseudocompact algebras.

In Section 3 we generalize the concept of a tilting module from the finite-dimensional case.
In this section we follow closely the Appendix in \cite{don1}.
To define tilting objects in a highest weight category ${\cal C}$, we need to define standard objects first.
But for that we need to have enough projective objects. Such a requirement is not always satisfied for ${\cal C}$. (See Remark \ref{remark32}.)
Nevertheless, it is possible that a full subcategory ${\cal C}[\Gamma]_f$ of ${\cal C}$ consisting of finite objects belonging to a finitely-generated ideal $\Gamma$ of the poset of weight $\Lambda$ of ${\cal C}$ has enough projectives.
For $\lambda\in \Gamma$, let $L(\lambda)$ be a simple object of highest weight $\lambda$ and $P_{\Gamma}(\lambda)$ be a projective cover of $L(\lambda)$ (in ${\cal C}[\Gamma]_f$).
Then we can define $\Delta(\lambda)$ as a largest quotient of $P_{\Gamma}(\lambda)$ whose composition factors have weights $\mu\leq\lambda$. 
Formally, this definition has a weak point since it depends on the choice of $\Gamma$.
To resolve all of these problems we assume that ${\cal C}$ admits a Chevalley duality. This guarantees the existence of projectives and standard objects in any ${\cal C}[\Gamma]$. Additionally, the definition of standard objects does not depend on the choice of $\Gamma$. 

Afterwards, we work together with left costandard and standard objects and homological properties of objects having increasing or decreasing costandard or standard filtrations. We define a tilting object as an object that has an increasing standard filtration and a decreasing costandard filtration (possibly both of them infinite). Let us note that, when doing modifications of the Appendix in \cite{don1}, we need to make substantial changes - say we use a different definition of a defect set, or the proof of the existence of tilting objects uses induction on more subtle parameters, and so on.

In Section 4, we consider the relationship between the left (discrete) standard modules and the right (pseudocompact) standard modules under the Chevalley duality. (In the case of the Chevalley duality, the right pseudocompact modules are also discrete modules because they are finite-dimensional.) We also characterize the Chevalley duality in terms of properties of the corresponding pseudocompact algebra. 

In Section 5, we define the Ringel dual $R$ of an ascending pseudocompact quasi-hereditary algebra $A$ with poset $\Gamma$ and analyze the properties of the corresponding Ringel functor and algebra $R$. In particular, $R$ is a descending pseudocompact quasi-hereditary algebra with respect to the $\Gamma$ considered with the opposite order. We believe that the Ringel functor is an equivalence of certain subcategories of costandardly and standardly filtered modules. Since the Ringel functor preserves extensions (see Theorem \ref{t52}), this is true for modules of finite length. In order to prove this in general, a symmetric theory for descending algebras is to be developed in the next article.

Finally, in Section 6, we consider an example of the general linear group $GL(1|1)$ in depth. First, we describe simple, costandard and injective $GL(1|1)$-supermodules completely; then we define the pseudocompact Schur superalgebra $S_r$ and describe its structure; and then we describe the Ringel dual of the related object $\Hat{S}_r$. We conclude the paper with topics for further investigation.
 
\section{Preliminaries}
We start by introducing definitions, notations and general assumptions with which we shall be working throughout the whole paper.

Let $K$ be an algebraically closed field and $\cal C$ be a $K$-abelian category. 
Suppose that $\cal C$ contains enough
injective modules and it is {\em locally artinian}, that is, $\cal C$
admits arbitrary direct unions of subobjects and any object is a
union of its subobjects of finite length (briefly, finite
subobjects). In particular, a composition factor $L$ of an object
$M$ is, by definition, a composition factor of a finite subobject of $M$.
The multiplicity of $L$ in $M$, denoted by $[M : L]$, is defined to
be the supremum of multiplicities $[S : L]$ over all finite
subobjects $S$ of $M$. Additionally, suppose that $\cal C$
satisfies the Grothendieck condition (condition AB5 in the terminology
of \cite{bd}). Then, by the final remark in Chapter 6, \S 3 of [6], any object in $\cal C$ has an injective envelope. 
Moreover, if $\phi: M\to N$ is an epimorphism and $K$ is a finite
subobject of $N$, then there is a finite subobject $K'$ of $M$ such that
$\phi |_{K'} : K'\to K$ is also an epimorphism. In particular, $\cal
C$ is a Grothendieck category with generators consisting of finite
subobjects.

For any poset $(\Gamma, \leq)$ and $\gamma\in\Gamma$, denote by $(\gamma]$ a (possibly
infinite) closed interval $\{\mu |\mu\leq\gamma\}$ and denote by $(\gamma)$ an open interval
$\{\mu |\mu <\gamma\}$. 

The category $\cal C$ is said to be a {\em highest weight category}
if there is an interval-finite poset $(\Lambda, \leq)$ (a set of highest
weights of $\cal C$) such that all non-isomorphic simple objects
are indexed by elements from $\Lambda$, say $\{L(\lambda)|\lambda\in\Lambda\}$, 
if there is a collection of {\em
costandard} objects $\{\nabla(\lambda)|\lambda\in\Lambda\}$ such that 
$L(\lambda)\subseteq\nabla(\lambda)$, and if
$[\nabla(\lambda)/L(\lambda) : L(\mu)]\neq 0$ implies $\mu
<\lambda$. Additionally, it is required that 
for any $\lambda, \mu\in\Lambda$, the dimensions 
$\dim_K Hom_{\cal C}(\nabla(\lambda), \nabla(\mu))$ and the multiplicities
$[\nabla(\lambda) : L(\mu)]$ are finite, and an injective envelope $I(\lambda)$ of
any $L(\lambda)$ has a (finite or infinite)\ {\em good (or costandard) filtration}
\[0\subseteq I_1\subseteq I_2\subseteq\ldots.\]
A filtration of $I$ is called {\em good} if 
$\bigcup_{n\geq 1}I_n=I(\lambda), I_1=\nabla(\lambda)$,
the factor $I_n/I_{n-1}$ is isomorphic to 
$\nabla(\mu_n)$ for a particular $\mu_n>\lambda$, 
and $\mu\in\Lambda$ equals $\mu_n$ for only finitely many indices $n$.
It follows that, for any objects of finite length $M, N\in \cal C$, the spaces
$Hom_{\cal C}(M, N)$ and $Ext^1_{\cal C}(M, N)$ are finite
dimensional. (See Lemma 3.2 of \cite{cps}.) Therefore $\cal C$ is of
finite type. (See \cite{tak}.)

Related to the concept of a highest weight category is a {\em quasi-hereditary} algebra. (See \cite{cps,dr,don1}.) 
Let $A$ be a $K$-algebra such that there is chain 
of right (or left) ideals 
\[0=H_0\subsetneq H_1\subsetneq\ldots\subsetneq H_n=A\]
of $A$ such that 
\begin{enumerate}
\item $H_n/H_{n-1}$ is a projective right (or left) $A/H_{n-1}$-module, 
\item $Hom_{A}(H_n/H_{n-1}, A/H_n)=0$, and
\item $Hom_{A}(H_n/H_{n-1}, \ rad (H_n/H_{n-1}))=0$. 
\end{enumerate}
\noindent Then $A$ is said to be a {\em quasi-hereditary} algebra and the above chain of ideals is called a {\em defining system of heredity ideals}.

For concepts, terminology, and basic results regarding pseudocompact algebras 
and discrete modules, we refer to \cite{br, gab, sim, v, vv}. A
$K$-algebra $R$ is called {\em pseudocompact} if $R$ is a complete
Hausdorff topological algebra with a basis $\{I\}$ of
neighborhoods at zero consisting of two-sided ideals of finite
codimension. In particular, $R$ is homeomorphic to the inverse
limit of finite-dimensional algebras $R/I$. A right $R$-module $M$
is said to be pseudocompact if $M$ is a complete Hausdorff
topological module with a basis $\{N\}$ of neighborhoods at zero
consisting of submodules of finite codimension. Again, $M$ is
homeomorphic to the inverse limit of finite-dimensional modules
$M/N$. The category $PC-R$ of pseudocompact right $R$-modules with
continuous homomorphisms is an abelian category with exact inverse
limits and with enough projective modules.

The dual category $R-Dis$ consists of all discrete left
$R$-modules. The duality functors $R-Dis\to PC-R$ and $PC-R\to
R-Dis$ are defined by $S\mapsto S^*=Hom_K(S, K)$ and $M\mapsto M^{\star}$,
respectively.
Here, $M^{\star}=hom_K(M, K)$ consists of every $\phi\in M^*$
such that ker $\phi$ contains an open submodule $N$ of $M$;
and a linear topology of $S^*$ is defined by submodules
$N^{\perp}=\{\phi\in S^* |\phi(N)=0\}$ for all finitely-generated submodules $N$ of $S$. 
The category $R-Dis$ is a locally
artinian Grothendieck category with enough injective modules. Its
generators are discrete finite-dimensional $R$-modules $R/I$,
where $I$ runs over all open left ideals of $R$. Denote by $R-dis$ a
full subcategory of $R-Dis$ consisting of all finite-dimensional modules.

Let $C$ be a $K$-coalgebra. We denote by $Comod-C$ the category of all
right $C$-comodules and by $comod-C$ a full subcategory of $Comod-C$
consisting of all finite-dimensional comodules. The dual
pseudocompact algebra $C^*$ has a basis of neighborhoods at zero
consisting of all two-sided ideals $D^{\perp}$ for all finite-dimensional subcoalgebras $D$ of $C$. 
There are isomorphisms of categories $Comod-C\simeq C^*-Dis$ and $comod-C\simeq
C^*-dis$ that identify objects as $K$-spaces. Under these isomorphisms, a $C$-comodule structure
of $M$, given by the coaction map $\tau_M(m)=\sum m_1\otimes c_2$ where $m, m_1\in M$ and  $c_2\in C$,
corresponds to a discrete left $C^*$-module structure given by $xm=\sum x(c_2)m_1$, where
$x\in C^*$. (See \cite{sim}.)
By a theorem of Takeuchi (see \cite{tak}), a highest weight category $\cal C$ is equivalent to $Comod-C$ for a coalgebra $C$. In particular,
${\cal C}\simeq C^*-Dis$.

Finally, for $\Gamma\subseteq\Lambda$ we can define functors $O^{\Gamma}$ and $O_{\Gamma}$. We can say that $M\in{\cal C}$
belongs to $\Gamma$ if and only if all composition factors $L(\lambda)$ of $M$ 
satisfy the condition that $\lambda\in\Gamma$. Any $N\in\cal C$ contains
a largest subobject $O_{\Gamma}(N)$ which belongs to $\Gamma$. 
Symmetrically, $N$ contains a unique minimal
subobject $O^{\Gamma}(N)$ such that $N/O^{\Gamma}(N)$ belongs to
$\Gamma$. For example, $\nabla(\lambda)=O_{(\lambda]}(I(\lambda))$ 
for each $\lambda\in\Lambda$.
The full subcategory of ${\cal C}$ consisting of all objects $M$ such that 
$O_{\Gamma}(M)=M$ is denoted by ${\cal C}[\Gamma]$. It is obvious
that $O_{\Gamma}$ is a left exact functor from $\cal C$ to ${\cal
C}[\Gamma]$ which commutes with direct sums. The functor
$O^{\Gamma} : {\cal C}\to {\cal C}$ also commutes with direct sums
but it is not right exact in general. In fact, for any exact
sequence
\[0\to X\to Y\to Z\to 0,\]
the map $O^{\Gamma}(Y)\to O^{\Gamma}(Z)$ is an epimorphism and
$O^{\Gamma}(X)\subseteq X\bigcap O^{\Gamma}(Y)$; but it is possible
that $O^{\Gamma}(X)$ is a proper subobject of
$O^{\Gamma}(Y)\bigcap X$.

\section{Pseudocompact quasi-hereditary algebras}

Let $B$ be a pseudocompact algebra. In what
follows, all pseudocompact $B$-modules are right modules and all discrete
$B$-modules are left modules unless otherwise stated.

We shall first derive certain properties of pseudocompact modules.

\subsection{Pseudocompact modules}

The following observation is useful.

\begin{lm}\label{lm1}
If $f : M\to N$ is a morphism in $PC-B$ and $V$ is a closed submodule of $M$, then $f(V)$ is
a closed submodule of $N$.
\end{lm}
\begin{proof}
Let $\{N_i | i\in I\}$ be a basis of neighborhoods at zero in the module $N$.
Then $\{M_i=f^{-1}(N_i) | i\in I\}$ is a family of open submodules in $M$
and $\ker f =\cap_{i\in I} M_i$. By Proposition 10 of Chapter IV, \S 3 of \cite{gab}, the image $f(V)$ is isomorphic to  
$\lim\limits_{\longleftarrow} (V+M_i/M_i)$. Since the map $f$ induces 
an isomorphism 
$\lim\limits_{\longleftarrow} (V+M_i/M_i)\simeq
\lim\limits_{\longleftarrow} (f(V)+N_i/N_i)=\overline{f(V)},$
we conclude that $f(V)= \overline{f(V)}$ and $f(V)$ is closed.
\end{proof}

As a consequence of the above lemma, we obtain that a sum of finitely many closed submodules of a
pseudocompact module $M$ is again a closed submodule of $M$, and any
finitely-generated $B$-submodule of $M$ is also closed. Additionally,
if $V$ is a closed submodule of $M$ and $J$ is a right closed ideal
of $R$, then \tc{$VJ$} is a closed submodule of $M$.

Denote the radical of $M$ by $rad M$, that is, $rad M$ is an intersection of all maximal
open submodules of $M$; and denote a factormodule $M/rad M$ by $top(M)$. 
It can be easily checked that $rad B$ is,
in fact, the Jacobson radical of $B$. Moreover, $rad B$ coincides with
an intersection of all open maximal two-sided ideals of $B$.
(See \cite{br}.) 

The radical of a projective module is described in the next lemma.

\begin{lm}\label{lm2}
If $P\in PC-B$ is a projective module, then $rad
P=P(rad B)$.
\end{lm}
\begin{proof} For any $M\in PC-B$, an inclusion $M(rad B)\subseteq rad M$ follows 
from Lemma 1.4 of \cite{br}. By Corollary 1 of Chapter IV, \S 3 of \cite{gab}, 
the projective module $P$ is isomorphic to $\prod_{i\in I} e_i B$, where 
$\{e_i|i\in I\}$ is a collection of (not necessarily different) primitive idempotents. If $N$ is a maximal open submodule of some $e_i B$, then
$\prod_{j\neq i} e_j B\bigoplus N$ is a maximal open submodule of
$P$. Therefore, it suffices to prove the statement of the lemma for every projective module $P= e_i B$; and this follows from 
Lemma 4.9 of \cite{v}.
\end{proof}

By Corollary 6.32 of \cite{bd}, every discrete
$B$-module has an injective envelope. Dually, every pseudocompact
$B$-module has a projective cover. Let $M$ be a simple
pseudocompact $B$-module. Then $M$ is a finite-dimensional and
discrete right $B$-module. If $P(M)\to M$ is a projective cover of $M$,
then its kernel is a unique maximal open submodule
of $P(M)$. In particular, this submodule coincides with $rad P(M)=
P(M)(rad B)$.

Since we shall be working with subspaces of type $eV$ and $Ve$, where $e$ is an idempotent of $B$,
it is important to observe the following relationship with closures.
If $M$ is a right (or left) pseudocompact $B$-module,
$V\subset M$ is a $K$-subspace of $M$, and
$e$ is an idempotent of $B$ such that $eV\subseteq V$ (or $Ve\subseteq V$, respectively), then
$\overline{eV}=e\overline{V}$ (or $\overline{Ve}=\overline{V}e$, respectively).

From now on, assume that $B=\prod_{i\in I} e_i B$, where $\{e_i|i\in I\}$ is a (summable) family of
primitive pairwise-orthogonal idempotents of the pseudocompact algebra $B$.
The simple modules and indecomposable projective modules are described below in terms of these idempotents.

\begin{pr}\label{pr1}
Any simple
pseudocompact $B$-module $M$ is isomorphic to $M(i)=top(e_i B)=e_i B/e_i (rad B)$ for some $i\in I$, and $P(M(i))\simeq e_i B$. Moreover,
$M(i)\simeq M(j)$ if and only if $e_i B\simeq e_j B$.
\end{pr}
\begin{proof} It follows from Lemma 3.9 of \cite{vv}.
\end{proof}

\begin{cor}\label{c1}
For every $i\in I$, the algebra $End_B(e_i B)$ is a local pseudocompact algebra.
\end{cor}
\begin{proof} For every $i\in I$, there is a natural isomorphism $End_B(e_i B)\simeq e_i Be_i$,
and $e_i Be_i$ is a closed subalgebra of $B$. Additionally, the condition that $\phi\in
End_B(e_i B)$ but $\phi$ is not an automorphism is equivalent to the condition
$Im \phi\subseteq e_i (rad B)$, which is, in turn, equivalent to the statement 
$\lim\limits_{n\to\infty} \phi^n=0$. Therefore, $rad
End_B(e_i B)\simeq e_i rad B e_i$, and $End_B(e_i B)/rad
End_B(e_i B)\simeq End_B(M(i))$ is a division algebra.
\end{proof}

Choose a subset $\Lambda$ of $I$, from the above decomposition
$B=\prod_{i\in I} e_i B$, 
whose elements are in one-to-one correspondence
with isomorphism classes of simple pseudocompact $B$-modules. 
A discrete $B$-module $L$ is simple if and only if $L^*\simeq
M(\lambda)$ for some $\lambda\in\Lambda$; in which case we denote
$L$ by $L(\lambda)$. If $I(\lambda)$ is an injective envelope of
$L(\lambda)$, then $I(\lambda)^*\simeq P(\lambda)=e_{\lambda} B$.
Moreover, $L(\lambda)$ is a composition factor of a discrete
$B$-module $N$ if and only if $M(\lambda)$ is a composition factor of
$N^*/K$ for some open submodule $K$ of $N^*$. In
particular, for any $S\in PC-B$, one can define the multiplicity of
$M(\lambda)$ in $S$, denoted $[S : M(\lambda)]$, as the supremum 
of all numbers $[S/K : M(\lambda)]$, where $[S/K:M(\lambda)]$ is the multiplicity of $M(\lambda)$ 
in a factormodule $S/K$, taken for every open submodule $K\subset S$.
Since $Hom_{PC-B}(P(\lambda), S)\simeq
Hom_{B-Dis}(S^{\star}, I(\lambda))$, we have $\dim_K
Hom_{PC-B}(P(\lambda), S)=\dim_K Se_{\lambda}=[S : M(\lambda)]=[S^{\star} : L(\lambda)]$.

Let $\Gamma$ be a subset of $\Lambda$. Analogously to Section 1, we say that a
pseudocompact $B$-module $S$ belongs to $\Gamma$ if $[S :
M(\lambda)]\neq 0$ implies $\lambda\in\Gamma$. For any $M\in PC-B$,
there exists a unique maximal closed submodule of $M$ which
belongs to $\Gamma$, which shall be denoted by $O_{\Gamma}(M)$. Also, there is a unique minimal closed
submodule $U$ of $M$ such that $M/U$ belongs to $\Gamma$. Such a module 
$U$ shall be denoted by $O^{\Gamma}(M)$. The modules $O_{\Gamma}(M)$ and $O^{\Gamma}(M)$ are described in the following lemma.

\begin{lm}\label{lm3}
Let $M$ be a pseudocompact $B$-module and $\Gamma\subseteq\Lambda$. Then 
\[O_{\Gamma}(M)=\overline{\sum_{N\subseteq M,\,
Ne_{\Lambda\setminus\Gamma}=0} N} \ \mbox{and} \
O^{\Gamma}(M)=Me_{\Lambda\setminus\Gamma}B,\]
where $e_{\Lambda\setminus\Gamma}=\sum_{\lambda\in\Lambda\setminus\Gamma}e_{\lambda}$.
\end{lm}
\begin{proof} By earlier remarks, $N\subseteq M$ belongs to $\Gamma$ if and only if
$Ne_{\Lambda\setminus\Gamma}=0$. Moreover, if $M/N$ belongs to
$\Gamma$, then $Me_{\Lambda\setminus\Gamma}B\subseteq N$. Since $Me_{\Lambda\setminus\Gamma}B$ is closed, our claim follows.
\end{proof}
By Lemma \ref{lm3}, $O^{\Gamma}(B)=Be_{\Lambda\setminus\Gamma}B$ is a
closed idempotent ideal of $A$ that shall be denoted by $H(\Gamma)$.
Moreover, for any $M\in PC-B$ we have $O^{\Gamma}(M)=H(\Gamma)M$.
(See also Appendix of \cite{don1}.)
Finally, observe that
\[O_{\Gamma}(M)^{\star}=M^{\star}/O^{\Gamma}(M^{\star}) \text{ and }
O^{\Gamma}(M)^{\star}=M^{\star}/O_{\Gamma}(M^{\star}),\] and
\[O_{\Gamma}(M^{\star})=(M/O^{\Gamma}(M))^{\star} \text{ and }
O^{\Gamma}(M^{\star})=(M/O_{\Gamma}(M))^{\star}.\]

For every $\lambda\in\Lambda$, we define a {\it standard} object
$\Delta(\lambda)=P(\lambda)/O^{(\lambda]}(P(\lambda))$. 
Then $\Delta(\lambda)^{\star}=\nabla(\lambda)$ and $P(\lambda)^\star=I(\lambda)$; and 
the following lemma is a
reformulation of properties of a highest weight category.

\begin{lm}\label{lm4}
The category $B-Dis$ is a highest weight category with respect to
an interval-finite poset $(\Lambda , \leq)$ if and only if every projective pseudocompact $B$-module $P(\lambda)$ has a
descending standard filtration (of closed submodules) where 
\[P(\lambda)=N_0\supseteq N_1\supseteq N_2\supseteq\ldots\]
such that $\bigcap_{k\geq 1}N_k=0$, and for every $k\geq 0$ we
have $N_{k}/N_{k+1}\simeq\Delta(\mu_k)$ where $\mu_k\geq\lambda$ and
$\mu_k=\lambda$ if and only if $k=0$. Additionally, for a given $\mu\in\Lambda$,
$\mu=\mu_n$ for only finitely many values of $n$; and also $\dim_K Hom_{PC-B}(\Delta(\lambda), \Delta(\mu))$ and
$[\Delta(\lambda) : M(\mu)]$ are finite for all $\lambda,
\mu\in\Lambda$.
\end{lm}

\subsection{Good posets}

From now on assume that $\Lambda$ is an interval-finite poset with respect to
a partial order $\leq$.

The following proposition relates a highest weight category $\cal C$ to each subcategory 
$\cal C[\Gamma]$ for every finitely-generated ideal $\Gamma$ of $\Lambda$.
 
\begin{pr}\label{pr3}
Let $\cal C$ be a locally-artinian category which satisfies the Grothendieck condition AB5. Assume that nonisomorphic simple objects of $\cal C$ 
are indexed by elements of an interval-finite poset $(\Lambda, \leq)$, $\cal C$ has enough injective objects, and every injective envelope of
a simple object belongs to a finitely or countably generated ideal of $\Lambda$. 

Then $\cal C$ is a highest weight category if and only if, for every
finitely-generated ideal $\Gamma\subseteq\Lambda$, the full
subcategory ${\cal C}[\Gamma]$ is a highest weight category.
\end{pr}
\begin{proof} The sufficient condition follows from Theorem 3.5 and Definition 3.1 (c)
of \cite{cps}. Conversely, let $L(\lambda)$ be a simple object in
$\cal C$ and $I=I(\lambda)$ be its injective envelope. If $\Gamma$
is a finitely-generated ideal containing $\lambda$, then
$I_{\Gamma}=O_{\Gamma}(I)$ is a finite injective envelope of
$L(\lambda)$ in ${\cal C}[\Gamma]$. Moreover, the costandard
object $\nabla_{\Gamma}(\lambda)$ in ${\cal C}[\Gamma]$ equals
$O_{(\lambda]}(I_{\Gamma})=O_{(\lambda]}(I)$, and therefore it
does not depend on the choice of $\Gamma$. Let us denote $\nabla_{\Gamma}(\lambda)$ simply by
$\nabla(\lambda)$. Each $\nabla(\lambda)$ obviously satisfies
condition (b) of Definition 3.1 of \cite{cps}. 
Assume $I$ belongs to an ideal
$\Gamma$ which is a union of an increasing chain of finitely-generated ideals, $\Gamma_i$ for all $i\geq 1$. By Theorem 3.5 of \cite{cps},
each $O_{\Gamma_i}(I)$ appears as a term in a finite good
filtration of $O_{\Gamma_j}(I)$, where $j >i$. 
For each $i\geq 1$ we can construct a good filtration of $O_{\Gamma_{i+1}}(I)/O_{\Gamma_i}(I)$ and combine it with
a given good filtration of $O_{\Gamma_i}(I)$. Proceeding this way step-by-step, we can construct a good filtration of $I$. 
The statement of the proposition shall follow if we show that if $\lambda\in\Gamma_i$, then
$\nabla(\lambda)$ does not appear as a factor of any good
filtration of $O_{\Gamma_j}(I)/O_{\Gamma_i}(I)$ for any $j > i$.
If $\nabla(\lambda)$ appears as a factor of a good
filtration of $O_{\Gamma_j}(I)/O_{\Gamma_i}(I)$ for some $j>i$, 
then Lemma 3.2(b) of \cite{cps} guarantees that there is a
good filtration of $O_{\Gamma_j}(I)/O_{\Gamma_i}(I)$ such that its first
factor is $\nabla(\mu)$, where $\mu\leq\lambda$. Then $\mu\in\Gamma_i$
and $O_{\Gamma_i}(I)$ is a proper subobject of an object
$N\subseteq O_{\Gamma_j}(I)$ which belongs to $\Gamma_i$. Thus we get a contradiction.
\end{proof}

As a consequence of the previous proposition, we shall concentrate on highest weight categories with 
certain finitely-generated/cogenerated posets of weights. Additionally, from now on, we shall assume that 
all costandard objects are finite.  

\begin{df}\label{d1}
Let $(\Gamma , \leq)$ be a poset and let $\lambda , \mu\in\Gamma$. We
call $\mu$ a {\it predecessor} of $\lambda$ if $\mu <\lambda$ and
there is no $\pi\in\Lambda$ such that $\mu <\pi <\lambda$. 
\end{df}

In what follows we shall denote a predecessor of $\lambda$ by $\lambda'$.

Now we are ready to formulate conditions on the poset $\Gamma$ that shall allow us to define the concept of an
ascending quasi-hereditary pseudocompact algebra later.

\begin{df}\label{d2}
A poset $(\Gamma, \leq)$ is said to be good if both of the following conditions are satisfied.

1) Each non-minimal element of $\Gamma$ has at least one but
only finitely many predecessors.

2) If $\mu <\lambda$, then there is a predecessor $\lambda'$ of $\lambda$ such that $\mu\leq\lambda'$.
\end{df}
If $\Gamma$ is interval-finite, then the second condition of the above definition is automatically satisfied.

The next proposition characterizes interval-finite good posets.

\begin{pr}\label{pr2}
A poset $(\Lambda, \leq)$ is interval-finite and good poset if and only if every finitely-generated ideal $\Gamma\subseteq\Lambda$  has a descending chain of finitely-generated subideals where
\[\Gamma=\Gamma_0\supseteq\Gamma_1\supseteq\Gamma_2\supseteq\ldots\]
such that $|\Gamma_k\setminus\Gamma_{k+1}| <\infty$ for every $k\geq 0$ and $\bigcap_{k\geq 0}\Gamma_k=\emptyset$. 
Moreover, the above chain of subideals can be chosen in such a way that the elements
of $\Gamma_k\setminus\Gamma_{k+1}$ are pairwise incomparable generators of $\Gamma_k$ for each $k\geq 0$.
\end{pr}
\begin{proof} Assume $(\Lambda, \leq)$ is an interval-finite and good poset and
$\Gamma=\bigcup_{1\leq i\leq t}(\lambda_i]$ is a finitely-generated ideal of $\Lambda$. 
Without a loss of generality we can assume that the elements $\lambda_i$ and $\lambda_j$ are
pairwise incomparable for $i\neq j$. If $\mu\leq\lambda_i$, then there is a chain
$\mu=\mu_0 <\mu_1 <\ldots <\mu_s=\lambda_i$ of a maximal length $s$ such that 
$\mu_j=\mu'_{j+1}$ for  $0\leq j\leq s-1$.
Denote the number $s$ by $ht_i(\mu)$ and set
$ht(\mu)=\max\limits_{i, \mu\leq\lambda_i} ht_i(\mu)$. It is
obvious that for any $k\geq 0$ a set $V_k=\{\mu\in\Gamma |
ht(\mu)=k\}$ is finite and a set $\Gamma_k=\{\mu\in\Gamma |
ht(\mu)\geq k\}$ coincides with a finitely-generated ideal
$\bigcup_{\mu\in V_k} (\mu]$. Since the elements of
$\Gamma_k\setminus\Gamma_{k+1}$ are pairwise incomparable,
the necessary condition follows. 

Conversely, suppose that the necessary
condition holds and assume that $\lambda$ is not a minimal element of
$\Lambda$. Take $\Gamma=(\lambda]$ and consider a chain
\[\Gamma=\Gamma_0\supseteq\Gamma_1\supseteq\Gamma_2\supseteq\ldots\]
as in the statement of the proposition, such that the elements of $\Gamma_k\setminus\Gamma_{k+1}$
are pairwise incomparable generators of $\Gamma_k$ for each $k\geq 0$.
If $\mu <\lambda$, then there is a $k > 0$ such that
$\mu\in\Gamma_k\setminus\Gamma_{k+1}$. Thus $[\mu
,\lambda]\subseteq\Gamma\setminus\Gamma_{k+1}$, which implies that $\Lambda$
is interval-finite. Finally, a predecessor of $\lambda$ either belongs to $\Gamma\setminus\Gamma_1$ or it is a generator of
$\Gamma_1$, and therefore $\Lambda$ is a good poset.
\end{proof}

In what follows we shall always assume that a chain
\[\Gamma=\Gamma_0\supseteq\Gamma_1\supseteq\Gamma_2\supseteq\ldots\]
is chosen so that the elements of $\Gamma_k\setminus\Gamma_{k+1}$ are pairwise incomparable
generators of $\Gamma_k$ for each $k\geq 0$.

\subsection{Ascending quasi-hereditary algebras}

The main concept of this part is that of an ascending quasi-hereditary pseudocompact algebra defined as below.

\begin{df}\label{d3}
Let $A$ be a pseudocompact algebra.
Suppose that all indecomposable projective $A$-modules are finite-dimensional and there is an ascending chain 
of closed ideals of $A$ where
\[0=H_0\subsetneq H_1\subsetneq\ldots\subsetneq H_n\subsetneq\ldots\]
such that 
\begin{enumerate}
\item \label{k1} for any open right ideal $I\subsetneq A$ there is an index $t$ such that $H_t\not\subseteq I$.

\noindent \hskip -0.53in Also suppose the following additional conditions are satisfied for every 

\noindent \hskip -0.53in $n\geq 1$.
\hskip 0.53in
\item \label{k2} $H_n/H_{n-1}$ is a projective pseudocompact $A/H_{n-1}$-module such that   
the number of its non-isomorphic indecomposable projective factors is finite.
\item \label{k3} $Hom_{PC-A}(H_n/H_{n-1}, A/H_n)=0$.
\item \label{k4}
$Hom_{PC-A}(H_n/H_{n-1}, \ rad (H_n/H_{n-1}))=0$. 
\end{enumerate}
Then $A$ is said to be an ascending quasi-hereditary pseudocompact algebra and the
above chain of ideals is called a defining system of ascending ideals.
\end{df}

The initial justification for discussing the above notion of an ascending quasi-hereditary algebra is provided by the following theorem. 
\begin{tr}\label{tr2}
If $A-Dis$ is a highest weight category with respect to a good finitely-generated poset
($\Gamma, \leq)$, then $A$ is an ascending quasi-hereditary pseudocompact
algebra with respect to a defining ascending system of closed ideals such that 
$H_n=H(\Gamma_n)$, where each $\Gamma_n$ is constructed as in Proposition \ref{pr2}. 
\end{tr}
\begin{proof}
It is obvious that a set $\{\mu |\mu\geq\lambda\}$ for any $\lambda\in\Gamma$ is finite.
Therefore every indecomposable injective module in $A-Dis$ has a finite good filtration. By our assumption, indecomposable 
injective modules are of finite length and are therefore finite-dimensional. Consequently, 
all indecomposable projective modules in $PC-A$ are finite-dimensional. 

According to Lemma \ref{lm3}, a descending chain of finitely-generated subideals $\Gamma_n$ of $\Gamma$ from Proposition \ref{pr2} yields an ascending chain of closed ideals $H_n$ of $A$. 

For a proper right open ideal $I$ of $A$, let $\{M(\lambda_i)\}_{1\leq i\leq
\ell}$ be a collection of all pairwise non-isomorphic composition
factors of $A/I$. 
Since $\bigcap_{k\geq0}\Gamma_k=\emptyset$, there is an index $t$ such that 
$\lambda_i\notin\Gamma_t$ for every $1\leq i\leq \ell$.
Then $H_t\not\subseteq I$, which proves condition (\ref{k1}) of Definition \ref{d3}.

Using Theorem 3.5 of \cite{cps}, we obtain that
$(A/H_n)-Dis=(A-Dis)[\Gamma_n]$ is a highest weight category
with respect to the poset $(\Gamma_n , \leq)$.
Therefore, we can proceed by induction on $n$ and it is enough to prove the
conditions (\ref{k2}), (\ref{k3}) and (\ref{k4}) of Definition \ref{d3} for $n=0$.

Let $A=\prod_{\lambda\in \Gamma} P(\lambda)^{m_{\lambda}}$,
where $m_{\lambda}$ can be infinite. Then
$$H_1=\prod_{\lambda\in\Lambda} O^{\Gamma_1}(P(\lambda))^{m_{\lambda}}.$$ 
By the dual of Lemma 3.2(b) and (v) from \cite{cps}, each module $P(\lambda)$ contains 
a submodule $M$ that is a direct sum of modules isomorphic to $\Delta(\mu)$, where $\mu\notin \Gamma_1$.
Since the factormodule $P(\lambda)/M$ is filtered by modules $\Delta(\pi)$, where $\pi\in\Gamma_1$, we 
get $O^{\Gamma_1}(P(\lambda))=M$. 
Therefore, $O^{\Gamma_1}(P(\lambda))=P(\lambda)=\Delta(\lambda)$ if $\lambda\notin \Gamma_1$,
and $O^{\Gamma_1}(P(\lambda))=0$ if $\lambda\in \Gamma_1$.
Thus each $O^{\Gamma_1}(P(\lambda))$ is a projective module; and consequently, so is $H_1$. 
Since there are only finitely many non-isomorphic indecomposable projective factors of $H(\Gamma_1)$, namely 
those isomorphic to modules $\Delta(\mu)$, where $\mu\notin\Gamma_1$,
condition (\ref{k2}) holds.

Condition (\ref{k3}) follows from the definition of functors $O^{\Gamma}$. 

If $\lambda\notin\Gamma_1$, then 
$[\Delta(\lambda) : M(\lambda)]=1$ implies $rad H_1$ belongs to $\Gamma_1$, that is 
$Hom_{PC-A}(H_1, rad H_1)=0$. This proves condition (\ref{k4}).
\end{proof}

Our next goal is to show that, vice-versa, if $A$ is an ascending quasi-hereditary algebra, 
then $A-Dis$ is a highest weight category.
Assume that $A$ is an ascending quasi-hereditary pseudocompact algebra and 
isomorphism classes of simple pseudocompact $A$-modules are indexed by elements $\lambda$ of a set $\Lambda$, and denote by
$M(\lambda)$ a simple pseudocompact $A$-module corresponding to $\lambda$.

 Then $M(\lambda)\simeq A/I$ for
an open right maximal ideal $I$ of $A$. According to condition (\ref{k1}) of Definition \ref{d3}, there is a maximal number
$k(\lambda)$ such that $H_{k(\lambda)}\subseteq I$.

The following pair of lemmas is of crucial technical importance.

\begin{lm}\label{lm5}
The index $k=k(\lambda)$, defined above, does not depend on a choice of the open right maximal ideal $I$ of $A$. 
Moreover, for any number $t$, there are only finitely many $\lambda\in\Lambda$ such that $k(\lambda)=t$.
\end{lm}
\begin{proof} 
Since $H_{k+1}+I=A$, $M(\lambda)$ is a direct factor of $top(H_{k+1}/H_k)$. Let $M(\lambda)\simeq A/I_1$ for another open
right ideal $I_1$ different from $I$. Let $k_1$ and $k_2$ be maximal numbers such that 
$H_{k_1}\subseteq I_1$ and $H_{k_2}\subseteq I_2=I\bigcap I_1$, respectively.  It is clear that $k,k_1\geq k_2$. 
Assume that $k > k_2$. Then $H_{{k_2}+1}+I_2=I$ and $M(\lambda)\simeq I/I_2$ is a direct factor of 
$top(H_{k_2 +1}/H_{k_2})$ by Lemma \ref{lm2}.
By condition (\ref{k3}) of Definition \ref{d3}, $M(\lambda)$ does not appear as a composition factor
of $A/H_{k_2 +1}$. On the other hand, $M(\lambda)$ is a factor of
$H_{k+1}/H_{k_2 +1}$ which is a contradiction implying that $k=k_2$. Symmetrically, we also get $k=k_1$. 

If $k(\lambda)=t$, then $M(\lambda)$ is a direct factor of $top(H_{t+1}/H_t)$. By condition (\ref{k2}) of Definition \ref{d3}, 
there are only finitely many such $\lambda\in \Lambda$.
\end{proof}

\begin{lm}\label{lm6}
For any $\lambda\in\Lambda$, a maximal index $k$, such that 
 $M(\lambda)$ is a composition factor of $A/H_k$, equals $k(\lambda)$.
\end{lm}
\begin{proof} Clearly, $M(\lambda)$ is a composition factor of $A/H_{k(\lambda)}$, which implies $k\geq k(\lambda)$.
If $M(\lambda)$ is a composition factor of $A/H_k$, then
there is an open right ideal $M\subseteq K$ of $A$ such that
$H_k\subseteq M$ and $K/M\simeq M(\lambda)$. Choose $x\in
K\setminus M$ and construct a continuous epimorphism $A\to K/M$ by
$a\mapsto xa + M$. The kernel of this epimorphism is an open ideal of $A$ containing $H_k$; hence by 
Lemma \ref{lm5} we obtain $k(\lambda)\geq k$.
\end{proof}

Define a partial order on $\Lambda$ by $\lambda\ < \mu$ if and only if
$k(\lambda) > k(\mu)$. 

The following theorem that generalizes Theorem 3.6 of \cite{cps} (see
also Proposition A3.7 from \cite{don1}) completes the description of the relationship 
of ascending quasi-hereditary pseudocompact algebras to highest weight categories.

\begin{tr}\label{tr1}
Let $A$ be an ascending quasi-hereditary pseudocompact algebra with a
defining system $\{H_n\}$ of ascending ideals. Then
$A-Dis$ is a highest weight category with respect to a poset
$(\Lambda,\leq)$, where the partial order $\leq$ defined above depends on $\{H_n\}$.
\end{tr}
\begin{proof} By Lemma \ref{lm5},  $(\Lambda, \leq)$ is an interval-finite and good poset. 
Using condition (\ref{k2}) of Definition \ref{d3}, we can write $H_1=\prod_{1\leq i\leq \ell} P(\lambda_i)^{m_i}$.
Then condition (\ref{k3}) implies $k(\lambda_i)=0$ for every $1\leq i\leq \ell$.
Conversely, assume that $k(\pi)=0$ for some $\pi\in\Lambda$ and 
write $M(\pi)=A/I$ for an open ideal $I$ of $A$.
Since $H_1\not\subseteq I$, we obtain that $M(\pi)$ is a direct summand of $top(H_1)$, and therefore 
$\pi=\lambda_i$ for some $1\leq i\leq \ell$.
We conclude that $\Lambda$ is generated by $\lambda_1, \ldots, \lambda_{\ell}$.

Define $\Lambda_1=\Lambda\setminus\{\lambda_1, \ldots, \lambda_{\ell}\}$.
We shall show that $H_1=H(\Lambda_1)$. Condition (\ref{k3}) of Definition \ref{d3} implies that $H(\Lambda_1)\subseteq H_1$.
If $H(\Lambda_1)\neq H_1$, then there is a proper open submodule $N$ of $H_1$ such that 
$H_1/N$ belongs to $\Lambda_1$. The module $H_1/N$ is a factormodule of a finite direct sum of projective modules of the form $P(\lambda_i)$
for some $1\leq i \leq \ell$.
Therefore one $M(\lambda_i)$ is a direct summand of $top(H_1/N)$; and this contradicts our assumption that
$H_1/N$ belongs to $\Lambda_1$. 

We can use induction on $t$ to show analogously that $H_t=H(\Lambda_t)$, where $\Lambda_t=\{\lambda|k(\lambda)\geq t\}$.

If $A=\prod_{\lambda\in\Lambda} P(\lambda)^{m_\lambda}$, then 
$\prod_{\lambda\in\Lambda} O^{\Lambda_1}(P(\lambda))^{m_{\lambda}}=H_1=\prod_{1\leq i\leq \ell} P(\lambda_i)^{m_i}$,
which shows that every $O^{\Gamma_1}(P(\lambda))$ is isomorphic to a
direct sum of copies of each $P(\lambda_i)$. 

Finally, condition (\ref{k4}) implies that the radical of $P(\lambda_i)$ does
not contain composition factors isomorphic to any $M(\lambda_j)$,
which means $P(\lambda_i)=\Delta(\lambda_i)$. 

Arguing by induction on
dimensions of indecomposable projective modules, we prove that all of them have required
finite $\Delta$-filtrations. Lemma \ref{lm4} concludes the proof.
\end{proof}

\subsection{Descending quasi-hereditary algebras}

In the definition of ascending quasi-hereditary algebra we have required an existence of an increasing filtration of closed ideals of $A$.
Now, in a similar fashion, we would like to consider a symmetric case when there is a descending filtration of closed ideals of $A$ where
\[A=G_0 \supsetneq G_1 \supsetneq \ldots \supsetneq G_n \supsetneq \ldots .\]

We start with a definition of a {\it successor} and a {\it cogood} poset.

\begin{df}\label{d6}
Let $(\Gamma , \leq)$ be a poset and let $\lambda , \mu\in\Gamma$. We
call $\mu$ a {\it successor} of $\lambda$ if and only if 
$\mu$ is a predecessor of $\lambda$ with respect to the opposite order $\leq^{op}$. 
\end{df}

In what follows we shall denote a successor of $\lambda$ by $'\lambda$.

The following symmetrical version of Definition \ref{d2} is needed in order to define descending quasi-hereditary pseudocompact algebras later.

\begin{df}\label{d7}
A poset $(\Gamma, \leq)$ is said to be 
cogood if the following conditions are satisfied.
\begin{enumerate}
\item  Each non-maximal element of $\Gamma$ has at least one but
only finitely many successors.
\item If $\mu > \lambda$, then
there is a successor $'\lambda$ of $\lambda$ such that $\mu\geq \ '\lambda$.
\end{enumerate}
\end{df}
Again, if $\Gamma$ is interval-finite, then the second condition of the above definition
is automatically satisfied. 

Clearly, the corresponding symmetrical version of  
Proposition \ref{pr2} holds.
\begin{pr}\label{pr4}
A poset $(\Lambda, \leq)$ is interval-finite and cogood if and only if every finitely-cogenerated coideal $\Gamma\subseteq\Lambda$  has a descending chain of finitely-cogenerated
subcoideals where 
\[\Gamma=\Gamma_0\supseteq\Gamma_1\supseteq\Gamma_2\supseteq\ldots\]
such that $|\Gamma_k\setminus\Gamma_{k+1}| <\infty$ for every $k\geq 0$ and $\cap_{k\geq 0}\Gamma_k = \emptyset$.
In this case coideals can be chosen such that elements of $\Gamma_k\setminus\Gamma_{k+1}$ are pairwise incomparable
cogenerators of $\Gamma_k$ for each $k\geq 0$.
\end{pr}

Now we are ready to introduce descending quasi-hereditary pseudocompact algebras.

\begin{df}\label{d8}
Let $A$ be a pseudocompact algebra.
Suppose that there is a descending chain 
of closed ideals of $A$ where
\[A=G_0\supsetneq G_1\supsetneq\ldots\supsetneq G_n\supsetneq\ldots\]
which satisfy the following conditions.

\begin{enumerate}
\item \label{n1} For any open right ideal $I$ of $A$, there is a non-negative number $n=n(I)$ such that $G_n\subseteq I$.
\item \label{n2} For every $n\geq 1$, $A/G_n$ is an ascending quasi-hereditary pseudocompact algebra with respect to a defining
system of ascending ideals, $\{H_k=G_{n-k}/G_n | 0\leq k\leq n\}$.
\end{enumerate}
Then $A$ is said to be a descending quasi-hereditary pseudocompact algebra and the
above chain of ideals is called a defining system of descending ideals.
\end{df}

Since the algebras of the form $A/G_n$ and a defining system of ideals of the form $H_k$ in the condition (\ref{n2}) above are finite, 
we could replace (\ref{n2}) by an equivalent condition.

(2') For every $n\geq 1$, $A/G_n$ is a quasi-hereditary algebra with respect to a defining
system of heredity ideals, $\{H_k=G_{n-k}/G_n | 0\leq k\leq n\}$.

The justification of a notion of a descending quasi-hereditary algebra is presented in the following theorem.

\begin{tr}\label{tr3}
A pseudocompact algebra $A$ is descending quasi-hereditary if and only if
$A-Dis$ is a highest weight category with respect to an interval-finite, cogood and finitely-cogenerated
poset $\Lambda$.
\end{tr}
\begin{proof} For the necessary condition, choose a descending chain of subcoideals such that 
$$\Lambda=\Lambda_0\supseteq\Lambda_1\supseteq\Lambda_2\supseteq\ldots$$
as in Proposition \ref{pr4}. For every $n\geq 0$, denote a finite ideal $\Lambda\setminus\Lambda_n$ by
$\Gamma_n$ and  set $G_n=H(\Gamma_n)$. Since every finite subset of $\Lambda$ is contained in some $\Gamma_n$, condition (\ref{n1}) of
Definition \ref{d8} holds. For condition (\ref{n2}), observe that every $\Gamma_n$ is
a finite good poset and apply Theorem \ref{tr2} to a highest weight category 
$(A/G_n)-Dis=(A-Dis)[\Gamma_n]$. 

Conversely, assume that $A$ is an descending quasi-hereditary pseudocompact algebra and 
isomorphism classes of simple pseudocompact $A$-modules are indexed by elements $\lambda$ of a set $\Lambda$, and denote by
$M(\lambda)$ a simple pseudocompact $A$-module corresponding to $\lambda$.
For each $n\geq 0$, define a subset $\Gamma_n$ of $\Lambda$ consisting of every $\lambda$
for which $M(\lambda)$ appears as a composition factor of $A/G_n$.
Clearly, $\emptyset=\Gamma_0\subseteq\Gamma_1\subseteq\ldots,$ and condition (\ref{n1}) of Definition \ref{d8} implies $\bigcup_{k\geq 0}\Gamma_k=\Lambda$.
By Theorem \ref{tr1}, each $\Gamma_n$ has a structure of a good poset, say $(\Gamma_n, \leq_n)$, with respect to which $(A/G_n)-Dis$ is
a highest weight category. 
Additionally, if $t\geq n$, then the partial order $\leq_n$ is a restriction of $\leq_t$ on $\Gamma_n$.
Moreover, $(A/G_n)-Dis=((A/G_t)-Dis)[\Gamma_n]$, which implies $G_n/G_t=H(\Gamma_n)/G_t$. Since $A$ is a Hausdorff topological space, 
we have $\bigcap_{k\geq 0}G_k  =0$ and 
\[G_n=\lim\limits_{{\longleftarrow}_{t\geq n}} G_n/G_t =\lim\limits_{{\longleftarrow}_{t\geq n}} H(\Gamma_n)/G_t=
H(\Gamma_n).\]
Define a partial order $\leq$ on $\Lambda$ via $\lim\limits_{\longrightarrow} (\Gamma_n, \leq_n)$. By the above considerations, elements of
$\Gamma_n\setminus\Gamma_{n-1}$ are pairwise incomparable generators of $\Gamma_n$ for every $n\geq 1$. 
Proposition \ref{pr4} implies that $\Lambda$ is an interval-finite, cogood and finitely-cogenerated poset. 
Since $(A/G_n)-Dis=(A-Dis)[\Gamma_n]$ for each $n\geq 0$, Proposition \ref{pr3} completes the proof. 
\end{proof}

\section{Tilting objects in a highest weight category}

\subsection{$\Delta$- and $\nabla$-filtrations}

Let $\cal C$ be a highest weight category with respect to a good poset of weights
$(\Lambda, \leq)$. In particular, if $\Gamma$ is a finitely-generated ideal of $\Lambda$, then
$\Gamma$ is countable. 
In what follows, we assume that all costandard objects are {\em Schurian}, that is, $End_{\cal C}(\nabla(\lambda))=K$ for
any $\lambda\in\Lambda$. 

According to \cite{cps} (see also Proposition \ref{pr3}), if $\Gamma$ is a finitely-generated ideal, then ${\cal C}[\Gamma]$
is a highest weight category with costandard objects of the form 
$\nabla(\lambda)$ and finite injective envelopes such that
$I_{\Gamma}(\lambda)=O_{\Gamma}(I(\lambda))$ for $\lambda\in\Gamma$. 
Theorem 3.9 of \cite{cps} states that $Ext^i_{\cal
C}(M, N)= Ext^i_{{\cal C}[\Gamma]}(M, N)$ for any $M, N\in {\cal
C}[\Gamma]$.

Let $\cal R$ be a class of objects from $\cal C$. The following definition is standard.
\begin{df}\label{d31}
An increasing filtration of an object $M=\bigcup_{i\geq 1}M_i$ wherein
\[0=M_0\subseteq M_1\subseteq M_2\subseteq\ldots,\]
 such that $M_i/M_{i-1}\in {\cal R}$ for every $i\geq 1$, 
is called an increasing ${\cal R}$-filtration. If $M$ has a decreasing filtration wherein
\[M=M_0\supseteq M_1\supseteq M_2\supseteq\ldots,\]
such that $M_i/M_{i+1}\in {\cal R}$ for every $i\geq 0$ and $\bigcap_{i\geq
0}M_i=0$, then it is called a decreasing ${\cal R}$-filtration.
\end{df}
For example, a good filtration of an injective envelope can be called a $\nabla$-filtration, where $\nabla=
\{\nabla(\lambda) |\lambda\in\Lambda\}$.

The next lemma is an easy generalization of a well-known result.

\begin{lm}\label{lm31}
Let $\Gamma$ be a finitely-generated ideal of $\Lambda$ and $M\in{\cal C}[\Gamma]$. If $Ext^i_{\cal C}(M, \nabla(\lambda))\neq 0$ for $i
> 0$, then there is a composition factor $L(\mu)$ of $M$ such
that $\mu >\lambda$.
\end{lm}
\begin{proof} Without a loss of generality, one can suppose that
$\lambda\in\Gamma$. Consider a short exact sequence
\[0\to\nabla(\lambda)\to I_{\Gamma}(\lambda)\to Q\to 0 ,\] where $Q$ has a finite $\nabla$-filtration with quotients of the form
$\nabla(\mu)$, where $\mu >\lambda$. A fragment 
\[Ext^{i-1}_{\cal C}(M, Q)\to Ext^i_{\cal C}(M, \nabla(\lambda))\to 0\]
of a corresponding long exact sequence
shows that $Ext^{i-1}_{\cal C}(M, Q)\neq 0$. An induction on the
index $i$ implies that $Hom_{\cal C}(M, \nabla(\mu))\neq 0$
for some $\mu
>\lambda$. Thus $L(\mu)$ is a composition factor of $M$.
\end{proof}

Denote by ${\cal C}_f$ a full subcategory of ${\cal C}$
consisting of all finite objects in ${\cal C}$. 

Our goal is to work towards a definition of a tilting object. In order to define a standard module $\Delta$, we need to make sure that we have 
enough projective objects.
From now on, we shall assume 
that for every finitely-generated ideal $\Gamma\subseteq\Lambda$ the subcategory ${\cal C}[\Gamma]_f$ has enough projectives. Additionally,
we shall assume that for every  $\lambda\in\Gamma$ the projective cover
$P_{\Gamma}(\lambda)$ of $L(\lambda)$ has a finite standard filtration, where the top quotient is $\Delta(\lambda)=
P_{\Gamma}(\lambda)/O^{(\lambda]}(P_{\Gamma}(\lambda))$ and every other quotient is $\Delta(\mu)$ for $\mu>\lambda$.
Besides, each $\Delta(\lambda)$ does not depend on the choice of an ideal $\Gamma$ such that $\lambda\in\Gamma$.
We shall denote the class $\{\Delta(\lambda) |\lambda\in\Lambda\}$ by $\Delta$.

These assumptions are satisfied if ${\cal C}_f$ has a {\em Chevalley duality}, that is, there is a duality $\tau : {\cal C}_f \to {\cal C}_f$ such that
$\tau(L(\lambda))\simeq L(\lambda)$ for each $\lambda\in \Lambda$.

\begin{lm}\label{lm32}
Assume ${\cal C}_f$ has a Chevalley duality. Then, for every finitely-generated ideal $\Gamma$ of $\Lambda$, the category ${\cal C}[\Gamma]_f$ has enough projective objects and each projective cover $P_{\Gamma}(\lambda)$ of $L(\lambda)$ has a finite standard filtration.
\end{lm}
\begin{proof} 
Applying the Chevalley duality, we obtain that $\Delta(\lambda)=\tau(\nabla(\lambda))$,
and every $P_{\Gamma}(\lambda)$ has a $\Delta$-filtration that is a Chevalley dual to the
$\nabla$-filtration of $I_{\Gamma}(\lambda)$ for each $\lambda\in\Gamma$.
\end{proof}
\begin{rem}\label{remark32}
The condition that every subcategory ${\cal C}[\Gamma]_f$ has enough projective objects does not imply that ${\cal C}$ has. 
For example, if ${\cal C}$ is a
category of rational $G$-modules over a reductive algebraic group $G$ (defined over an algebraically closed field of positive characteristic), then 
it contains a projective module if and only if $G$ is a finite extension of a torus. See \cite{don3}.

Let us also remark that if ${\cal C}=R-Dis$ for a pseudocompact algebra
$R$, then the modules $P_{\Gamma}(\lambda)$ and $\Delta(\lambda)$
differ from the projective and standard modules defined in Section 2.
In fact, they belong to different categories. Nevertheless, these objects are related and we shall discuss their relationship in 
Section 4.
\end{rem}

The following lemma is formulated for standard objects and it is a symmetric variant of Lemma \ref{lm31} .
\begin{lm}\label{lm33}
Let $\Gamma$ be a finitely-generated ideal of $\Lambda$ and $M\in {\cal
C}[\Gamma]$. If $Ext^i_{\cal C}(\Delta(\lambda), M)\neq 0$ for $i
> 0$, then there is a composition factor $L(\mu)$ of $M$ with $\mu
>\lambda$.
\end{lm}
Combining Lemmas \ref{lm31} and \ref{lm32}, we obtain the following corollary.
\begin{cor}(Compare this with Theorem 3.11 of \cite{cps}.)\label{c31}
For every $\lambda, \mu\in\Lambda$ and $i > 0$, we have $Ext^i_{\cal C}(\Delta(\lambda), \nabla(\mu))=0$.
\end{cor}

\begin{df}\label{d33}
If $M\in{\cal C}$, then the set $\{\lambda | Ext^1_{\cal C}(\Delta(\lambda), M)\neq 0\}$ is denoted by $S(M)$ and called
the defect set of $M$.
\end{df}
Let us mention that our definition is different from the definition in \cite{don1}. In
fact, we do not require a defect set to be an ideal of $\Lambda$.

In the future we shall usually limit our considerations to restricted objects as defined below.

\begin{df}\label{d32}
An object $M\in {\cal C}[\Gamma]$ is said to be $\Gamma$-restricted
if $[M : L(\lambda)]$ is finite for every $\lambda\in\Gamma$.
\end{df}

From now on, $\Gamma$ shall be a finitely-generated ideal of $\Lambda$ unless
stated otherwise. If it does not lead to confusion, we shall omit
a subindex $\Gamma$ in our notation.

Existence of a $\Delta$ (or a $\nabla$)-filtration relates to the vanishing of certain extensions in the following
theorem that generalizes Corollary \ref{c31}.

\begin{tr}\label{t31}
Assume that $M$ is $\Gamma$-restricted. If $M$ has an increasing (or decreasing)
$\Delta$-filtration, then
$Ext^i_{\cal C}(M, \nabla(\lambda))=0$ for every $\lambda\in\Lambda$
and $i\geq 1$. If $M$ has an increasing
(or decreasing)  $\nabla$-filtration, then $Ext^i_{\cal
C}(\Delta(\lambda), M)=0$ for every $\lambda\in\Lambda$ and $i\geq
1$.
\end{tr}
\begin{proof} We consider only the case of a decreasing $\Delta$-filtration; all other cases are similar.
We can assume that $\lambda\in\Gamma=\bigcup_{j=1}^{k}(\pi_j]$. Since a set $A=\bigcup_{j=1}^{k}(\lambda , \pi_j]$
is finite, there is a finite subobject $N\subseteq M$ such that $[N : L(\nu)]=[M : L(\nu)]$ for every
$\nu\in A$. 
Let $M_k$ denote the $k$-th member
of a decreasing $\Delta$-filtration of $M$.
Then $N\bigcap M_k=0$ for a sufficiently large $k$. In this case $[M_k : L(\nu)]=0$ for every $\nu\in A$
and $Ext^i_{\cal C}(M_k, \nabla(\lambda))=0$ by
Lemma \ref{lm31}. Since $Ext^i_{\cal C}(M/M_k, \nabla(\lambda))=0$ by
Corollary \ref{c31}, the exactness of the fragment
\[Ext^i_{\cal C}(M/M_k, \nabla(\lambda))\to Ext^i_{\cal C}(M, \nabla(\lambda))
\to Ext^i_{\cal C}(M_k, \nabla(\lambda))\]
of the long exact sequence implies that $Ext^i_{\cal C}(M,\nabla(\lambda))=0$.
\end{proof}

The next lemma shows that it is enough to state the previous theorem only for $i=1$ since the vanishing of 
first extensions implies the vanishing of all higher extensions.

\begin{lm}\label{l34}
If $M$ belongs to $\Gamma$ and
$Ext^1_{\cal C}(\Delta(\lambda), M)=0$ for every $\lambda\in\Lambda$
(or $Ext^1_{\cal C}(M,\nabla(\lambda))=0$ for every $\lambda\in\Lambda$),
then $Ext^i_{\cal
C}(\Delta(\lambda), M)=0$ (or $Ext^i_{\cal C}(M,
\nabla(\lambda))=0$, respectively) for every
$\lambda\in\Lambda$ and $i > 1$.
\end{lm}
\begin{proof} Assume that $Ext^i_{\cal C}(\Delta(\lambda), M)\neq 0$ 
for a weight $\lambda\in\Lambda$ and for $i > 1$.
Lemma \ref{lm33} implies that $\lambda\in\Gamma$ and we can work in
the category ${\cal C}[\Gamma]$. 
A short exact sequence
$$0\to Q\to P(\lambda)\to \Delta(\lambda)\to 0$$
induces an exact fragment
$$Ext^{i-1}_{\cal C}(Q, M)\to Ext^i_{\cal C}(\Delta(\lambda), M)\to 0 .$$
Since $Q$ has a $\Delta$-filtration with each factor $\Delta(\mu)$, where $\mu >\lambda$, one can argue by induction
on $i$ to get a contradiction. The proof of the second statement is analogous.
\end{proof}

The following theorem is a partial converse to Theorem \ref{t31}.

\begin{tr}\label{t32}
Let $M$ be a $\Gamma$-restricted object.
If $Ext^1_{\cal C}(\Delta(\lambda), M)=0$ for every $\lambda\in\Lambda$, 
then $M$ has a decreasing $\nabla$-filtration. 
If $Ext^1_{\cal C}(M, \nabla(\lambda))=0$ for every $\lambda\in\Lambda$, then $M$ has an increasing
$\Delta$-filtration.
\end{tr}
\begin{proof} Suppose first that $Ext^1_{\cal C}(\Delta(\lambda), M)=0$ for all
$\lambda\in\Lambda$. Corresponding to a decreasing chain
of finitely-generated ideals of $\Lambda$ where
$$\Gamma=\Gamma_0\supseteq\Gamma_1\supseteq\ldots$$
such that $\Gamma\setminus\Gamma_k$ is finite for every $k\geq 0$
and $\bigcap_{k\geq 0}\Gamma_k =\emptyset$, 
there is a
decreasing chain of subobjects where
$$M=M_0\supseteq M_1\supseteq M_2\supseteq\ldots$$
such that $M_k=O_{\Gamma_k}(M)$,
every quotient $M/M_k$ is finite, and $\bigcap_{0\leq k}M_k=0$. 

The socle of every quotient $M/M_k$ belongs to $\Gamma\setminus\Gamma_k$, and each $M/M_k$ can be embedded into
a finite sum of finite indecomposable injective objects from ${\cal C}[\Gamma]$.

Consider the following fragment 
$$Hom_{\cal C}(\Delta(\mu), M/M_k)\to Ext^1_{\cal C}(\Delta(\mu), M_k)\to 0\to $$
$$\to Ext^1_{\cal C}(\Delta(\mu), M/M_k)\to Ext^2_{\cal C}(\Delta(\mu), M_k)$$
of the long exact sequence
Assume $Ext^1_{\cal C}(\Delta(\mu), M_k)\neq 0$. Then $\mu\in\Gamma_k$ by Lemma \ref{lm33}. 
On the other hand, the socle of $M/M_k$ belongs to $\Gamma\setminus\Gamma_k$, which implies $Hom_{\cal C}(\Delta(\mu),
M/M_k)=0$. Thus we have a contradiction. Therefore $Ext^1_{\cal C}(\Delta(\mu), M_k)=0$ for
every $\mu$. By Lemma \ref{l34}, $Ext^2_{\cal C}(\Delta(\mu),
M_k)=0$ for every $\mu$, which implies $Ext^1_{\cal
C}(\Delta(\mu), M/M_k)=0$ for every $\mu$. Finally, since every object $M/M_k$
is finite, we conclude the proof by standard arguments from
\cite{dr,jan}.

For the second statement, it is enough to prove that every subobject $O^{\Gamma_k}(M)$ is finite. In fact, $O^{\Gamma_k}(M)$
contains a finite subobject $N$ such that $[N : L(\mu)]=[O^{\Gamma_k}(M) : L(\mu)]$ for every $\mu\in\Gamma\setminus\Gamma_k$
because all the multiplicities are finite. In particular, $O^{\Gamma_k}(M)/N$ belongs to $\Gamma_k$, that is, $N=O^{\Gamma_k}(M)$.
The final argument is analogous to the proof of the first part of this theorem.
\end{proof}

As a consequence of Theorems \ref{t31} and \ref{t32} we obtain the following corollary.

\begin{cor}\label{c32}
Assume that $M$ is $\Gamma$-restricted
and that there is an exact sequence
$$0\to N\to M\to S\to 0 .$$
Then the following are true.
\begin{enumerate}
\item If both $M$ and $N$ have decreasing $\nabla$-filtrations,
then $S$ has a decreasing $\nabla$-filtration. 
\item If both $M$
and $S$ have increasing $\Delta$-filtrations, then $N$ has an
increasing $\Delta$-filtration. 
\item If $M$ has a decreasing
(increasing) $\nabla$ ($\Delta$)-filtration, then every direct
summand of $M$ has a decreasing (increasing) $\nabla$
($\Delta$)-filtration. 
\item If $M$ has a decreasing
$\nabla$-filtration, then the object $\nabla(\lambda)$ appears
exactly $(M : \nabla(\lambda))=\dim Hom_{\cal C}(\Delta(\lambda),
M)$ times as a section of the $\nabla$-filtration of $M$. 
Additionally, $(M : \nabla(\lambda))=(N :
\nabla(\lambda))+(S : \nabla(\lambda))$, provided that $N$ has a
decreasing $\nabla$-filtration. 
\item If $M$ has an increasing
$\Delta$-filtration, then the object $\Delta(\lambda)$ appears
exactly $(M :\Delta(\lambda))=\dim Hom_{\cal C}(M,
\nabla(\lambda))$ times as a section of the $\Delta$-filtration of $M$. 
Moreover, $(M :
\Delta(\lambda))=(N : \Delta(\lambda))+(S : \Delta(\lambda))$,
provided that $S$ has an increasing $\Delta$-filtration.
\end{enumerate}
\end{cor}

\subsection{Tilting objects}

Now we are ready to define a tilting object in our setting.

\begin{df}\label{d34}
An object $T\in{\cal C}$ is called a tilting object if and only if it has an increasing $\Delta$-filtration and 
$Ext^1_{\cal C}(\Delta(\lambda), T)=0$ for every $\lambda\in\Lambda$.
\end{df}

The existence of indecomposable titling modules is established in the following theorem.

\begin{tr}\label{t33}
For any weight $\lambda\in\Lambda$, there is an indecomposable tilting object $T$ such that $T$ is $(\lambda]$-restricted,
$[T : L(\lambda)]=1$, and its
$\Delta$-filtration begins with $\Delta(\lambda)$.
\end{tr}
\begin{proof} Fix a decreasing chain of ideals of $\Lambda$ as before; say
\[\Gamma_0=(\lambda]\supseteq\Gamma_1\supseteq\Gamma_2\supseteq\ldots\]
such that $\Gamma_k\setminus\Gamma_{k+1}$ is finite for all $k\geq
0$, $\bigcap_{k\geq 0}\Gamma_k=\emptyset$, and the elements of $\Gamma_{k}\setminus \Gamma_{k+1}$ 
are pairwise incomparable generators of $\Gamma_k$.
For any $V\in{\cal C}$ we shall write $O^k(V)$ for $O^{\Gamma_k}(V)$.

We shall construct an increasing chain of finite objects where
$$0=V_0\subseteq V_1\subseteq V_2\subseteq\ldots$$
such that
\begin{enumerate}
\item $V_k/V_{k-1}$ has a $\Delta$-filtration, 
\item $O^t(V_k)=V_t$ for all $k\geq t$,
\item $V_k$ is indecomposable, and
\item $S(V_k)\subseteq \Gamma_k$ for all $k\geq 1$.
\end{enumerate}
It is clear that $V_1=\Delta(\lambda)$ satisfies the inductive hypothesis. Suppose
that we have already constructed a fragment of our filtration up
to the $k$-th term. Consider every finite object $X$,
belonging to $\Gamma_0=(\lambda]$, which satisfies the following
conditions:
\begin{enumerate}
\item $V_k\subseteq X$, 
\item$X/V_k$ has a $\Delta$-filtration, 
\item $O^{k+1}(X)=X$ and $O^t(X)=V_t$ for all $t\leq k$,
\item $X$ is indecomposable, and 
\item $S(X)\subseteq \Gamma_k$.
\end{enumerate}
Any such object $X$ shall be called $k$-{\it admissible}. It is obvious that $V_k$ is
$k$-admissible.
 
Denote $D_k=\Gamma_k\setminus\Gamma_{k+1}=\{\mu_1, \ldots, \mu_l\}$ and choose a linear order 
$\mu_1 >\mu_2 >\ldots >\mu_l$ on $D_k$.
Define a characteristic of a $k$-admissible object $X$ as $\chi(X)=(\chi_{\mu}(X))_{\mu\in D_k}$
\[=(\dim_K Ext^1_{\cal C}(\Delta(\mu_1), X),\ldots , \dim_K Ext^1_{\cal C}(\Delta(\mu_l), X)).\]
We introduce an order on characteristics of $k$-admissible objects by considering them as vectors of a
poset ${\mathbb N}^k$ with respect to the lexicographical order from left to right.

Choose a $k$-admissible object $X$ with a minimal characteristic. If
$S(X)\subset \Gamma_{k+1}$, then the set $V_{k+1}=X$.

Otherwise, take a minimal element
$\pi\in S(X)\setminus \Gamma_{k+1}$ and a non-split exact sequence
$$0\to X\to Y\to \Delta(\pi)\to 0,$$
where the object $Y$ belongs to $\Gamma_0=(\lambda]$. 
We shall show that $O^{k+1}(Y)=Y$ and $O^t(Y)=V_t$ for all $t\leq k$. Indeed, since no
non-zero factor of $\Delta(\pi)$ belongs to $\Gamma_{k+1}$,
we obtain that $O^{k+1}(Y)+X=Y$. 
Since $X=O^{k+1}(X)$, we get $X\subseteq O^{k+1}(Y)$, and therefore $Y=O^{k+1}(Y)$. 
Next, $\pi\in \Gamma_k$ implies $O^t(\Delta(\pi))=0$ for all $t\leq k$. Hence $O^t(Y)=O^t(X)=V_t$ for
all $t\leq k$.

Suppose that $Y=Y_1\bigoplus Y_2 \bigoplus \ldots \bigoplus Y_l$, where every $Y_i$ is an indecomposable object.
Then $V_k=O^k(Y_1)\bigoplus O^k(Y_2)\bigoplus \ldots \bigoplus O^k(Y_l)$ and all but one of these summands are zeros because
$V_k$ is indecomposable. Assume  $O^k(Y_2)=\ldots =O^k(Y_l)=0$. Then $O^t(Y_2)=\ldots =O^t(Y_l)=0$ and $O^t(Y_1)=V_t$  for $t\leq k$. 
Moreover, it is obvious that $O^{k+1}(Y_i)=Y_i$ for all $i=1,\ldots, l$. Finally, $Y/V_k=Y_1/V_k\bigoplus Y_2\bigoplus \ldots \bigoplus Y_l$ and 
$Y_1/V_k$ has a $\Delta$-filtration.

Take any $\mu\in\Lambda$ and consider the following fragment 
\[0\to Hom_{\cal C}(\Delta(\mu), X)\to Hom_{\cal C}(\Delta(\mu), Y)\to Hom_{\cal C}(\Delta(\mu), \Delta(\pi))\]
\[\to Ext^1_{\cal C}(\Delta(\mu), X)\to Ext^1_{\cal C}(\Delta(\mu), Y)\to Ext^1_{\cal C}(\Delta(\mu), \Delta(\pi))\]
of the long exact sequence.

Suppose that $\mu\in S(X)\setminus\Gamma_{k+1}$. If $\mu\neq\pi$,
then $Hom_{\cal C}(\Delta(\mu), \Delta(\pi))= Ext^1_{\cal
C}(\Delta(\mu), \Delta(\pi))=0$ because $\pi$ is minimal in $S(X)\setminus \Gamma_{k+1}$. In this case
$\chi_{\mu}(X)=\chi_{\mu}(Y)$.
If $\mu =\pi$, then one can give a verbatim proof of Lemma A4.1 from
\cite{don1} to derive that 
$\chi_{\pi}(Y)=\chi_{\pi}(X)-1$. 

Finally, if $\mu\in S(Y)\setminus
S(X)$, then $Ext^1_{\cal C}(\Delta(\mu), \Delta(\pi))\neq 0$, and
therefore $\mu <\pi$.

The above shows that $S(Y_1)\subseteq
S(Y)\subseteq\Gamma_k$ and $Y_1$ is $k$-admissible but
$\chi(Y_1)\leq\chi(Y)<\chi(X)$. This contradiction shows that $S(X)\subseteq \Gamma_{k+1}$ for a $k$-admissible object $X$ 
with a minimal characteristic.

Let $T=\bigcup_{1\leq k}V_k$. It is obvious that $T$ is $(\lambda]$-restricted.
Using Lemma 3.8(a) from \cite{cps}, we obtain that $Ext^1_{\cal C}(\Delta(\mu), T)=0$ for every $\mu\in\Lambda$.
To conclude the proof, we must recognize that $O^k(T)=V_k$ for all $k\geq 1$.
\end{proof}

From now on we shall denote the tilting object from the above theorem by
$T(\lambda)$. It is obvious that
$O^{1}(T(\lambda))=O^{(\lambda)}(T(\lambda))=\Delta(\lambda)$, and 
we shall call this subobject of $T(\lambda)$  a {\em standard bottom} of $T(\lambda)$.
By Theorem \ref{t32}, $T(\lambda)$ has a decreasing filtration,
$T(\lambda)/O_{(\lambda)}(T(\lambda))=\nabla(\lambda)$; and we shall call 
this factorobject of $T(\lambda)$ a {\em costandard top} of $T(\lambda)$.

Next, two technical lemmas shall be needed in the proofs of the following theorems.

\begin{lm}\label{l35}
Assume $Y\in{\cal C}$ is $\Gamma$-restricted and $Y$ has a filtration 
\[0=Y_0\subseteq Y_1 \subseteq Y_2 \subseteq \ldots \]
such that each factor $Y_{k}/Y_{k-1}$, for all $k\geq 1$, has a finite $\Delta$-filtration.
If $Z\in{\cal C}$ satisfies $Ext^1(\Delta(\mu), Z)=0$ for every $\mu\in \Lambda$, then
every morphism $Y_t\to Z$ can be extended to a morphism $Y\to Z$ for each $t\geq 0$.
\end{lm}
\begin{proof}
Clearly $Y=\lim\limits_{\longrightarrow} Y_k$, which implies $Hom_{\cal C}(Y, Z)=\lim\limits_{\longleftarrow} Hom_{\cal C}(Y_k,Z)$.
Since $Y_m/Y_t$ has a finite $\Delta$-filtration for each $m\geq t$, we infer that $Ext^1(Y_m/Y_t,Z)=0$.
Therefore every restriction morphism  $Hom_{\cal C}(Y_m, Z)\to Hom_{\cal C}(Y_t, Z)$ is an epimorphism, and our claim follows.
\end{proof}

The following lemma is from folklore.
\begin{lm}\label{l36}
Let $\{A_t, \phi_{mt} : A_t\to A_m |1\leq m\leq t\}$ be a
projective spectrum of finite-dimensional local algebras, where each
$\phi_{mt}$ is an epimorphism. Then a pseudocompact algebra
$A=\lim\limits_{\longleftarrow} A_t$ is also local.
\end{lm}
\begin{proof} Since the ground field $K$ is algebraically closed, it follows that $A_t/rad A_t=K$ for every $t\geq 1$. In particular,
$\phi_{mt}(rad A_t)=rad A_m$ for every $1\leq m\leq t$. Since an inverse limit ${\cal M}$ of the induced spectrum of radicals
is pronilpotent and $A/{\cal M}=K$, the claim follows.
\end{proof}

The basic properties of $End_{\cal C}(T(\lambda))$ are described in the following theorem.

\begin{tr}\label{t34}(Also see Proposition \ref{p51} later.)
An algebra $A(\lambda)=End_{\cal C}(T(\lambda))$ is a local
pseudocompact algebra and $A(\lambda)/rad A(\lambda)=K$.
Moreover, $\phi\in rad A(\lambda)$ if and only if a restriction of $\phi$ to the
standard bottom of $T(\lambda)$ is zero or, equivalently, $\phi$
induces the zero endomorphism of the costandard top of $T(\lambda)$.
\end{tr}
\begin{proof}
Let $T(\lambda)=\bigcup_{k\geq 1} V_k$ as in the proof of Theorem \ref{t33}.
Then $A(\lambda)=\lim\limits_{\longleftarrow} Hom_{\cal C}(V_k,T(\lambda))$ and 
$Hom_{\cal C}(V_k,T(\lambda))=End_{\cal C}(V_k)$ for every $k\geq 0$.
By Lemma \ref{l35}, for every $m\leq k$, a map $\phi_{km} : Hom_{\cal C}(V_k,T(\lambda))\to
Hom_{\cal C}(V_m,T(\lambda))$, induced by an inclusion $V_m \to V_k$, is an epimorphism of algebras. 
To see that $End_{\cal C}(V_k)$ is a local algebra, observe that 
$End_{\cal C}(V_1)=End_{\cal C}(\Delta(\lambda))=K$ and that $\psi_k\in End_{\cal C}(V_k)$ belongs to $rad(End_{\cal C}(V_k))$ if and only if
$\phi_{k1}(\psi_k)=0$.

Since $End_{\cal C}(\nabla(\lambda))=K$, a map $\phi\in End_{\cal C}(T(\lambda))$
induces either an automorphism or a zero morphism of the costandard
top of $T(\lambda)$. On the other hand, $L(\lambda)=\Delta(\lambda) +
O_{(\lambda)}(T(\lambda))/O_{(\lambda)}(T(\lambda))$ is the socle
of $\nabla(\lambda)$. Therefore, in the first case, 
$\phi|_{\Delta(\lambda)}\neq 0$ and $\phi$ is invertible, while in the second case
$\phi|_{\Delta(\lambda)}=0$.
\end{proof}

\begin{cor}\label{c33}
Let $T$ be a tilting object such that it has an increasing $\Delta$-filtration beginning with $\Delta(\lambda)$.
Then $T(\lambda)$ is a direct summand of $T$. In particular,
$T(\lambda)$ is uniquely defined by the weight $\lambda$ up to an isomorphism.
\end{cor}
\begin{proof} Since both $T$ and $T(\lambda)$ have increasing $\Delta$-filtrations that start with $\Delta(\lambda)$, 
the identifications of these two copies of $\Delta(\lambda)$ can by Lemma \ref{l35} be extended to the morphisms 
$\phi : T(\lambda)\to T$ and $\psi : T\to T(\lambda)$. Then $\alpha=\psi\phi$ is an
automorphism of $T(\lambda)$ because $\psi\phi |_{V_1}=id$; and therefore,
$\alpha^{-1}\psi:T\to T(\lambda)$ splits.
\end{proof}

We refine the above corollary in the next theorem which generalizes a well-known result from the classical case.
\begin{tr}(See Theorem A4.2 of \cite{don1}.)\label{t35}
Let $T$ be a $\Gamma$-restricted tilting object. Then $T$ equals a direct (possibly infinite) sum of
some indecomposable tilting subobjects of the form $T(\lambda)$.
\end{tr}
\begin{proof} Let 
\[\Gamma=\Gamma_0\supseteq\Gamma_1\supseteq\Gamma_2\supseteq\ldots\]
be a decreasing chain of ideals of $\Lambda$
such that $|\Gamma_k\setminus\Gamma_{k+1}| <\infty$ for every $k\geq 0$, $\bigcap_{k\geq 0}\Gamma_k=\emptyset$,
and the elements of $\Gamma_k\setminus\Gamma_{k+1}$ are pairwise incomparable
generators of $\Gamma_k$ for each $k\geq 0$.
Then $T$ has a filtration
$$0=T_0\subseteq  T_1\subseteq T_2\subseteq\ldots ,$$
where $T_k=O^{\Gamma_k}(T)$ for $k\geq 0$. 
The set $\Gamma_0\setminus\Gamma_1=\{\mu_1, \ldots, \mu_s\}$ consists of maximal weights of $\Gamma$.
By Lemma \ref{lm33} there is a
$\Delta$-filtration of $T$ which begins with
$$\bigoplus_{1\leq i\leq s}\Delta(\mu_i)^{(T_1 :\Delta(\mu_i))}.$$
Corollary \ref{c33} implies that
$T=S_1\bigoplus T^{(1)}$, where
\[S_1=\bigoplus_{1\leq i\leq s}T(\mu_i)^{(T_1 :\Delta(\mu_i))}\]
and $T^{(1)}$ is a $\Gamma$-restricted tilting object
such that $O^{\Gamma_1}(T^{(1)})=0$ and $(T^{(1)} :\Delta(\mu_i))=0$ for any $1\leq i\leq s$.
In particular, $[T : L(\lambda)]=[S_1 : L(\lambda)]$ for every $\lambda\in\Gamma_0\setminus
\Gamma_1$.

In the next step, we decompose $T^{(1)}$ as $S_2\bigoplus T^{(2)}$, where $S_2$ is a finite direct sum of
indecomposable tilting subobjects and $O^{\Gamma_2}(T^{(2)})=0$. In particular, $[T^{(1)} : L(\lambda)]=[S_2 : L(\lambda)]$
for every $\lambda\in\Gamma_1\setminus
\Gamma_2$.

We can continue with similar decompositions as necessary. 

Finally, for $S=\bigoplus_{1\leq i} S_i$ we have $[T : L(\lambda)]=[S : L(\lambda)]$ for every $\lambda\in\Gamma$, which means $S=T$.
\end{proof}

The last theorem of this section extends an important property from the classical case.

\begin{tr}\label{t36}
Let $\Gamma$ and $\widetilde{\Gamma}$ be finitely-generated ideals of $\Lambda$, $T$ be a $\Gamma$-restricted tilting object, and $M$ be a
$\widetilde{\Gamma}$-restricted object. If $M$ has an increasing (or decreasing) $\nabla$-filtration,
then $Ext^1_{\cal C}(T, M)=0$. If $M$ has an increasing (or decreasing) $\Delta$-filtration and $T$ is a finite direct sum of
indecomposable tilting objects, then $Ext^1_{\cal C}(M, T)=0$.
\end{tr}
\begin{proof} For the first statement, by Theorem \ref{t35} it is enough to show that $Ext^1_{\cal
C}(T(\lambda), M)=0$ for all $\lambda\in\Lambda$. Consider a
short exact sequence
$$0\to M\to X\stackrel{\pi}{\to} T(\lambda)\to 0 .$$
Since $Ext^1_{\cal C}(\Delta(\lambda), M)=0$ by Theorem \ref{t31}, the object $X$
contains a subobject which is isomorphic to the first member
$\Delta(\lambda)$ of a $\Delta$-filtration of $T(\lambda)$.
Moreover, this isomorphism is induced by the epimorphism $\pi$.
Since $Ext^1_{\cal C}(\Delta(\mu), X)=0$ for any
$\mu\in\Lambda$, Lemma \ref{l35} implies that there is a morphism $\phi :
T(\lambda)\to X$ such that $\pi\phi$ is an isomorphism. Thus
$\phi$ is an inclusion and $X=M\bigoplus Im \phi$, where $Im \phi \simeq T(\lambda)$.

For the second statement, it is enough to prove that  $Ext^1_{\cal C}(M, T(\lambda))=0$ for every $\lambda\in\Lambda$.
Consider an exact sequence
$$0\to T(\lambda)\stackrel{\iota}{\to} X\to M\to 0.$$
Then Theorem \ref{t31} implies that $Ext^1_{\cal C}(X,\nabla(\mu))=0$ for every $\mu\in\Lambda$.
For the $\widetilde{\Gamma} \bigcup (\lambda]$-restricted
object $X'=X/\Delta(\lambda)$, there is a short exact sequence
$$0\to T(\lambda)/\Delta(\lambda)\to X'\to M\to 0$$
and a fragment 
\[0=Ext^1_{\cal C}(X,\nabla(\mu))\to Ext^1_{\cal C}(X',\nabla(\mu))\to Ext^2_{\cal C}(\Delta(\lambda),\nabla(\mu))\]
of the corresponding long exact sequence for every $\mu\in\Lambda$. By Corollary \ref{c31}, the last expression in the above fragment vanishes, and this implies $Ext^1(X',\nabla(\mu))=0$ for every $\mu\in\Lambda$.
By Theorem \ref{t32}, $X'$ has an increasing $\Delta$-filtration. 
Therefore, $X$ has an increasing $\Delta$-filtration which begins with $\Delta(\lambda)$.
By Lemma \ref{l35} there is a morphism $\psi : X\to T(\lambda)$ which extends the inclusion
$\Delta(\lambda)\to T(\lambda)$. Therefore, $\psi\iota$ is an isomorphism and $X=T(\lambda)\bigoplus Ker\psi$, where
$Ker\psi\simeq M$.
\end{proof}
\begin{cor}\label{c34}
If $T$ and $T'$ are restricted tilting objects (possibly corresponding to different finitely-generated ideals $\Gamma$ and $\Gamma'$ of $\Lambda$), 
then $Ext^1_{\cal C}(T, T')=0$.
\end{cor}

\section{Chevalley duality}

The goal of this section is to describe the Chevalley duality for pseudocompact algebras. 
The whole discussion can be reduced to the case when the algebra is basic.

\begin{df}\label{41}
Let $B$ be a pseudocompact algebra. If $B=\prod_{i\in I} P(i)$, where
every direct factor $P(i)$ is an indecomposable projective $B$-module and
$P(i)\not\simeq P(j)$ for $i\neq j$, then $B$ is called {\it basic}.
\end{df}
\begin{lm}\label{l41}(See also Corollary 5.5 and Proposition 5.6 from \cite{sim}.) 
For every pseudocompact algebra $B$ there is a basic pseudocompact algebra $A$ such that
a category $PC-B$ is equivalent to a category $PC-A$.
\end{lm}
\begin{proof} Write 
$B=\prod_{\lambda\in\Lambda}P(\lambda)^{m_{\lambda}}$, 
where $P(\lambda)^{m_{\lambda}}$ is a direct (possibly infinite) product of all projective indecomposable factors of $B$, 
and $e_{\lambda} B=P(\lambda)=P(M(\lambda))$ for $\lambda\in\Lambda$; and set
$e=\sum_{\lambda\in\Lambda}e_{\lambda}$. 
We shall show that a Schur functor $F$ from the category $PC-B$ to a category $PC-eBe$, 
given by $M\mapsto Hom_{PC-B}(eB, M)=Me$, is full, faithful
and dense. 

To show that $F$ is full and faithful, note that every $M\in PC-B$ has a
{\it projective presentation}
$$Q_1\to Q_0\to M\to 0$$
such that $Q_0$ and $Q_1$ are direct factors of $(eB)^I$ (for a possibly
infinite index set $I$). Additionally, for any direct
product of projective modules $\prod_{j\in J}e_j B$ (where an idempotent $e_j$ is
not necessarily primitive and can be repeated), a space $Hom_{PC-B}(\prod_{j\in
J}e_j B, M)$ is isomorphic to a subspace of $\prod_{j\in J} M e_j$
consisting of all elements of the form $\prod_{j\in J} m_j$ such that the
collection $\{m_j |j\in J\}$ is summable in $M$.
Then arguments as in the proof of Proposition 2.5 of Chapter II of \cite{ars}, show that $F$ is full and faithful.

If $S\in PC-eBe$, then $S\simeq (S\hat{\otimes}_{eBe} eB)e$, where $\hat{\otimes}$ denotes the complete
tensor product. (See \cite{br}.) Therefore $F$ is dense. Theorem 1.2 of Chapter II in \cite{ars}
shows that $F$ is an equivalence of categories.
\end{proof}

Let $A$ be a basic pseudocompact algebra and let $\Lambda$ be a set whose
elements are in one-to-one correspondence with isomorphism classes of
simple pseudocompact $A$-modules. Fix a
decomposition $A=\prod_{\lambda\in\Lambda}e_{\lambda}A$, where $e_{\lambda}A=P(\lambda)$ 
is an indecomposable projective
factor of $A$ such that its top is isomorphic to a simple $A$-module $M(\lambda)$. 

If all indecomposable
projective modules $e_{\lambda}A$ are finite-dimensional and $\Lambda$
is at most countable, then $A$ is called a {\it restricted} pseudocompact
algebra. It is clear that every ascending quasi-hereditary pseudocompact
algebra is restricted. For the remainder of this section, we assume that
$A$ is restricted.

Consider a decreasing chain of subsets wherein
\[(*) \phantom{***********} \Lambda=\Gamma_0\supseteq\Gamma_1\supseteq\Gamma_2\supseteq\ldots \phantom{***********} \]
\noindent such that $\Omega_k=\Lambda\setminus\Gamma_k$ has cardinality $k$
for all $k\geq 0$, and $\bigcap_{k\geq 0}\Gamma_k=\emptyset$.
There is a
decreasing chain of two-sided closed ideals of $A$ wherein
$$A=J_0\supseteq J_1\supseteq J_2\supseteq\ldots$$
such that $\bigcap_{k\geq 0} J_k=0$,
where
$J_k=H(\Omega_k)$ for every $k\geq 0$.
We say that this chain of closed
ideals of $A$ is defined by the above chain of sets $(*)$ or, briefly, that it
is a $(*)$-chain. 
The description of the Chevalley duality shall be formulated for such $(*)$-chains.

If $A$ is basic, then every $J_k$ is open. In fact,
$A/J_k\simeq \prod_{\lambda\in\Omega_k}P(\lambda)/J_k P(\lambda)$
is finite-dimensional for every $k\geq 0$.
The structure of $A/J_k$ is clarified in the following lemma.

\begin{lm}\label{lm42}
For every $k\geq 0$, the idempotents $\{e_{\lambda}|\lambda\in\Omega_k\}$ form a complete family of pairwise
orthogonal primitive idempotents of $A$ modulo the ideal $J_k$.
\end{lm}
\begin{proof} We have $e_{\lambda}A\bigcap J_k=e_{\lambda}J_k=
O^{\Omega_k}(e_{\lambda}A)\subseteq rad (e_{\lambda}A)$ for
every $\lambda\in\Omega_k$. Thus $(rad (e_{\lambda}A)+J_k)/J_k=
rad (e_{\lambda}A+J_k/J_k)$ is a unique maximal open submodule
of $(e_{\lambda}A+J_k)/J_k$. Therefore, $(e_{\lambda}A+J_k)/J_k$ is indecomposable, and $e_{\lambda}$ is a
primitive idempotent modulo $J_k$. The remaining assertions clearly follow.
\end{proof}

If there is a continuous anti-isomorphism $\phi : A\to A$, then one
can define a duality $\tau_{\phi} : (A-Dis)_f\to (A-Dis)_f$ by
$M\mapsto\tau_{\phi}(M)=M^*$, where $M^*$ is a left discrete
$A$-module via $(af)(m)=f(\phi(a)m)$ for $a\in A$ and $f\in M^*$.

A category $(A-Dis)_f[\Omega_k]=(A/J_k-Dis)_f$
consists of all finite-dimensional $A/J_k$-modules. If $\tau$ is a
Chevalley duality on $(A-Dis)_f$, then  $\tau_k=\tau
|_{(A/J_k-Dis)_f}$ is a Chevalley duality on $(A/J_k-Dis)_f$.
Moreover, $(A/J_k-Dis)$ is a full subcategory of $(A/J_t-Dis)$ for every
$t\geq k$ and $\tau_t |_{(A/J_k-Dis)_f}=\tau_k$.

Denote the opposite pseudocompact algebra of $A$ by $A^{\circ}$. We have
a natural anti-equivalence  $(A-Dis)_f\to (A^{\circ}-Dis)_f$ given by $M\mapsto M^*$,
where $M^*$ is a left discrete $A^{\circ}$-module via
$(af)(m)=f(a m)$ for $a\in A^{\circ}$ and $f\in M^*$. Composing this anti-equivalence with $\tau$,
we obtain an equivalence $\tau' : (A-Dis)_f\to (A^{\circ}-Dis)_f$.
It is obvious that a restriction of $\tau'$ on each $(A/J_k-Dis)_f$
induces an equivalence $\tau'_k : (A/J_k-Dis)_f\to
(A^{\circ}/J_k-Dis)_f$ compatible with the full embedding
$(A/J_k-Dis)_f\subseteq (A/J_t-Dis)_f$ for each $t\geq k$.

The next lemma allows us to proceed from equivalence of categories of finite-dimensional modules to isomorphisms
of underlying algebras.

\begin{lm}\label{lm43}
Let $\pi : (A-mod)_f\to (B-mod)_f$ be an equivalence of categories. If \ $\dim
A=\dim B <\infty$ and $A$ or $B$ is basic, then there is an
isomorphism $\phi : B\to A$ such that $\pi\simeq\bar\phi$, where
$\bar\phi (M)=M$ is a left $B$-module via $bm=\phi(b) m$ for $m\in M,
b\in B$ and $M\in A-mod$.
\end{lm}
\begin{proof} In fact, standard Morita arguments work within
subcategories of finitely-generated modules. (See  \S 22 of \cite{un}.) In particular, if $\alpha$ is an 
equivalence inverse to $\pi$, then there are natural isomorphisms of the form 
$$\pi(M)\simeq Hom_B(B , \pi(M))\simeq Hom_A(\alpha(B), M)$$
of $B$-modules, where the last space has a left $B$-module
structure given by $bf=f\alpha(r_b)$ for $f\in Hom_A(\alpha(B), M)$ and $b\in B$, and where
$r_b\in End_B(_{B}B)$ is the right multiplication operator
$r_b(x)=xb$ for $b, x\in B$. Assume that $A$ is basic (The case wherein $B$ is basic follows similarly.) 
Assume $A=\prod_{1\leq i\leq s}Ae_i $ with each factor $Ae_i$ pairwise non-isomorphic to $Ae_j$ for $i\neq j$ and 
with $Ae_i$ indecomposable.
 By Theorem 22.1 of \cite{un}, $\alpha(B)$ is a
faithfully balanced $(B, A)$-bimodule and also an
$A$-progenerator. Therefore $\alpha(B)=\prod_{1\leq i\leq
s}(Ae_i)^{m_i}$, where each $m_i$ is non-zero. The functor $\alpha$
induces an isomorphism $\phi : B\simeq End_B({_B}B)\to
End_A(_{A}\alpha(B))$. If at least one $m_i$ is greater than $1$,
then $\dim B > \dim A$, which implies a contradiction. Hence $_{A}\alpha(B)=_{A}A$, and $\phi :
B\to End_A(_{A}A)=A$ is the required isomorphism.
\end{proof}

\begin{cor}\label{c41}
There is an isomorphism $\phi_k : A/J_k\to A^{\circ}/J_k$ induced
by the equivalence $\tau'_k$ for every $k\geq 0$. Moreover, for every $t\geq
k$, we have $\bar\phi_t|_{A/J_k-Dis}=\bar\phi_k$. In
particular, $\phi_t(J_k/J_t)=J_k/J_t$, and a restriction of $\phi_t$ to
$A/J_k$ is an isomorphism that coincides with $\phi_k$.
\end{cor}
\begin{proof} For any $M\in (A/J_k-Dis)_f$, there is a commutative diagram:
$$\begin{array}{ccccccc}
\tau'_t(M) &\simeq & Hom_{A^{\circ}/J_t}(A^{\circ}/J_t , \tau'(M))
& \simeq & Hom_{A/J_t}(A/J_t , M) &\simeq & M \\
\Vert & &\Vert &&\Vert &&\Vert \\
\tau'_k(M) &\simeq & Hom_{A^{\circ}/J_k}(A^{\circ}/J_k , \tau'(M))
& \simeq & Hom_{A/J_k}(A/J_k , M) &\simeq & M
\end{array}.$$
From this diagram we infer that $M\phi_t(J_k/J_t)=0$. The claim follows when 
we consider $M=A/J_k$.
\end{proof}

The following theorem gives a characterization of the Chevalley duality for pseudocompact algebras.

\begin{tr}\label{t41}
If $A$ is a restricted and basic pseudocompact algebra, then $(A-Dis)_f$ possesses a
Chevalley duality $\tau$ if and only if there is a continuous
anti-isomorphism $\phi$ of $A$, which preserves a $(*)$-chain
such that $\tau_{\phi}\simeq \tau$.
\end{tr}
\begin{proof} 
If $(A-Dis)_f$ has a Chevalley duality $\tau$, then 
Corollary \ref{c41} implies that $\phi=\lim\limits_{\leftarrow}\phi_k$
defines an anti-isomorphism of $A$ such that $\phi(J_k)= J_k$ for 
$k\geq 0$. Clearly $\tau_{\phi}\simeq \tau$.

Assume now that an anti-isomorphism $\phi$ of $A$ preserves a $(*)$-chain.
Then $\tau_{\phi}$ is a duality on $(A-Dis)_f$. If
$L=L(\pi)$ is a simple discrete $A$-module, then
$\pi\in\Omega_k\setminus\Omega_{k-1}$ if and only if $J_k L=0$ and
$J_{k-1}L=L$ for every $k\geq 1$.  Since $J_k\tau_{\phi}(L)=0$ and $J_{k-1}\tau_{\phi}(L)=\tau_{\phi}(L)$ for 
every $k\geq 1$, we conclude that $\tau_{\phi}(L)\simeq L$.
\end{proof}
\begin{rem}\label{rm41}
If $\phi$ preserves one $(*)$-chain, then it preserves every such chain.
\end{rem}
\begin{rem}\label{rm42}
Assume that $A$ is a basic ascending quasi-hereditary pseudocompact algebra and $(A-Dis)_f$ has
a Chevalley duality $\tau$. Then we can identify $\tau$ and $\tau_{\phi}$ as in Theorem \ref{t41}. 
If $I(\lambda)$ is a (finite-dimensional) indecomposable injective envelope of a simple module 
$L(\lambda)\in A-Dis$,
then the (right) projective pseudocompact $A$-module $P(\lambda)=I(\lambda)^*$ is finite-dimensional,
and therefore $P(\lambda)$ is discrete. As a vector space, $P(\lambda)$  coincides with $\tau_{\phi}(I(\lambda))$ 
and its structure as a left discrete $A$-module is given via the anti-isomorphism $\phi$.
Analogously, corresponding to a (left) costandard $A$-discrete module, we obtain a (right) standard pseudocompact $A$-module, 
which is also a (right) discrete $A$-module. With the help of the anti-isomorphism $\phi$, we can view this module as a left discrete $A$-module.
\end{rem}

\section{A Ringel dual}
Let $A$ be an ascending pseudocompact quasi-hereditary algebra. By Theorem \ref{tr1}, $\cal C=A-Dis$ is a highest weight category with respect to a
good finitely-generated poset $(\Gamma, \leq)$. Fix a decreasing chain of subideals
$$\Gamma=\Gamma_0\supseteq\Gamma_1\supseteq\Gamma_2\supseteq\ldots$$
as in Proposition \ref{pr2}. 

Assume that $\cal C$ has a
Chevalley duality $\tau$. Without a loss of generality one can assume that
$\tau=\tau_{\phi}$, where $\phi$ is a continuous anti-isomorphism of $A$ and $A$ is basic. As in Section 3, there is a collection
of indecomposable tilting (left discrete) modules
$\{T(\lambda)\}_{\lambda\in\Gamma}$. Let
$T=\bigoplus_{\lambda\in\Gamma} T(\lambda)$ be a {\it complete tilting module}. It is clear that $T$
is $\Gamma$-restricted. Denote the algebra $End_{\cal
C}(T)^{\circ}$ by $R$ and call it {\it the Ringel dual} of $A$. 
Let ${\cal C}(\nabla)$ denote a 
full subcategory of $\cal C$ consisting of all restricted modules
having a decreasing $\nabla$-filtration. 
By Theorem \ref{t32}, for every $k$, both
$O^{\Gamma_k}(T)=\bigoplus_{\lambda\in\Gamma\setminus\Gamma_k}O^{\Gamma_k}(T(\lambda))$
and $T/O^{\Gamma_k}(T)$ have increasing $\Delta$-filtrations. Additionally, $O^{\Gamma_k}(T)$ is a finite-dimensional module for every $k$.

\subsection{Properties of the Ringel dual}

We shall require the following lemma.

\begin{lm}\label{l51}
If $M\in {\cal C}(\nabla)$, then $Ext^1_{\cal
C}(T/O^{\Gamma_k}(T), M)=0$ for every $k\geq 0$.
\end{lm}
\begin{proof} Denote $T/O^{\Gamma_k}(T)=D_k$ and consider the following fragment 
\[Hom_{\cal C}(O^{\Gamma_k}(T), M)\to Ext^1_{\cal C}(D_k,M)\to Ext^1(T,M)\] 
of the long exact sequence.
By Theorem \ref{t36}, the third term vanishes. If $M$ belongs to $\Gamma_k$, then the first term equals zero, and consequently,
$Ext^1_{\cal C}(D_k,M)=0$.

Now consider the following fragment 
\[Ext^1_{\cal C}(D_k, O_{\Gamma_k}(M))\to Ext^1_{\cal C}(D_k, M)\to Ext^1_{\cal C}(D_k, M/O_{\Gamma_k}(M))\]
of the long exact sequence.
The first term vanishes because $O_{\Gamma_k}(M)$ belongs to $\Gamma_k$.  
By Theorem \ref{t31}, we have $Ext^1_{\cal C}(\Delta(\lambda),M)=0$ for every $\lambda\in \Lambda$, and 
we can argue, as in the proof of Theorem \ref{t32}, that 
$M/O_{\Gamma_k}(M)$ has a finite $\nabla$-filtration. Since $D_k$ has a finite $\Delta$-filtration, we can use Theorem \ref{t31} again to derive that
$Ext^1(D_k, M/O_{\Gamma_k}(M)=0$. The exactness of the above fragment concludes the proof.
\end{proof}

There is a right exact functor $F : {\cal C}\to R-mod$ defined by
$F(M)=Hom_{\cal C}(T, M)$ for $M\in\cal C$. The algebra  $R$ has a
linear topology defined by two-sided ideals $I_k=Hom_{\cal
C}(T/O^{\Gamma_k}(T), T)$ for each $k\geq 0$. Moreover, an $R$-module $F(M)$
has a linear topology defined by (right) $R$-submodules $F(M)_k=Hom_{\cal
C}(T/O^{\Gamma_k}(T), M)$ for each $k\geq 0$. 

The following proposition extends results of Theorem \ref{t34} and characterizes functor $F$ on ${\cal C}(\nabla)$.

\begin{pr}\label{p51}
If $M\in {\cal C}(\nabla)$, then a canonical homomorphism
$$F(M)\to\lim\limits_{\leftarrow}
F(M)/F(M)_k$$ is an isomorphism. In particular, $R$ is a
pseudocompact algebra and $F$ is an exact functor from $\cal{C}(\nabla)$ to $R-PC$.
\end{pr}
\begin{proof} By Lemma \ref{l51}, projective spectrums $\{F(M)/F(M)_k| k\geq
0\}$ and $\{Hom_{\cal C}(O^{\Gamma_k}(T), M) | k\geq 0\}$ are
canonically isomorphic and
$F(M)\simeq\lim\limits_{\leftarrow}Hom_{\cal C}(O^{\Gamma_k}(T),
M)$. $F(M)$ is a pseudocompact $R$-module because $I_k F(M)=F(M)_k$ and
$\dim_K Hom_{\cal C}(O^{\Gamma_k}(T), M) <\infty$ for each $k$ since $M$ is
restricted. By Theorem \ref{t36}, $F$ is exact.
\end{proof}
There is a decomposition $R\simeq \prod_{\lambda\in\Gamma} F(T(\lambda))=\prod_{\lambda\in\Gamma}R\hat{e}_{\lambda}$,
where each $\hat{e}_{\lambda}$ is a canonical projection $T\to T(\lambda)$. 
Since $\lambda\in \Gamma_k$ implies $\hat{e}_{\lambda}\in I_k$,
a set $\{\hat{e}_{\lambda}\}_{\lambda\in\Gamma}$ is a summable collection of idempotents of $R$. Moreover, 
$End_R(R\hat{e}_{\lambda})\simeq \hat{e}_{\lambda}R\hat{e}_{\lambda}\simeq End_{\cal C}(T(\lambda))$ is a local pseudocompact algebra by Theorem \ref{t34}. 
Therefore $\{\hat{P}(\lambda)=R\hat{e}_{\lambda}\}_{\lambda\in\Gamma}$ is a complete set of pairwise non-isomorphic indecomposable projective $R$-modules.

For each $\lambda\in\Lambda$, denote $top (\hat{P}(\lambda))$ by $\hat{M}(\lambda)$ and $F(\nabla(\lambda))$ by $\hat{\Delta}(\lambda)$.
Then $End_R(\hat{M}(\lambda))=K$ by Corollary \ref{c1}.

The next lemma establishes a reciprocity relationship.

\begin{lm}\label{lm52}(See Lemma A4.6 of \cite{don1}.)
For every $\lambda, \mu\in \Gamma$ we have $$[\hat{\Delta}(\lambda) : \hat{M}(\mu)]=(T(\mu) : \Delta(\lambda)).$$
In particular, every $\hat{\Delta}(\lambda)$ is finite-dimensional.
\end{lm}
\begin{proof}
Since
$$[\hat{\Delta}(\lambda) : \hat{M}(\mu)]=\dim\hat{\Delta}(\lambda)\hat{e}_{\mu}=\dim Hom_{\cal C}(T(\mu), \nabla(\lambda)),
$$
part (5) of Corollary \ref{c32} concludes the proof.\end{proof}

The first main result of this section states that the Ringel dual of an ascending quasi-hereditary pseudocompact algebra is a descending
quasi-hereditary pseudocompact algebra.

\begin{tr}\label{tr51} (See Theorem A4.7 of \cite{don1}.)
The pseudocompact algebra $R$ is a descending quasi-hereditary algebra with respect to the poset $(\Gamma, \leq^{op})$.
\end{tr}
\begin{proof}
Let $T(\lambda)=T_0\supseteq T_1\supseteq\ldots$ be a decreasing $\nabla$-filtration of $T(\lambda)$ such that $T_0/T_1\simeq\nabla(\lambda)$.  
For every $k\geq 0$, denote $F(T_k)$ by $N_k$. Each $N_k$ is a closed submodule of $\hat{P}(\lambda)$ and $\bigcap_{k\geq 0} N_k=0$. 
By Theorem \ref{t36}, $Ext^1_{\cal C}(T, T_k)=0$ for every $k\geq 0$. Therefore, $N_k/N_{k+1}\simeq F(T_k/T_{k+1})\simeq \hat{\Delta}(\mu_k)$. 
Moreover, $\mu_0=\lambda$ and $\mu_k >^{op}\lambda$ whenever $k\geq 1$.
Finally, Lemma \ref{lm52} implies that $\hat{\Delta}(\lambda)=\hat{P}(\lambda)/O^{[\lambda)^{op}}(\hat{P}(\lambda))$.
\end{proof}

For a pseudocompact algebra $R$, let $R-PC(\hat{\Delta})$ denote a full subcategory of $R-PC$ that consists of all restricted (pseudocompact) modules having a decreasing $\hat{\Delta}$-filtration. Using arguments from the proof of Theorem \ref{tr51}, we infer that $F$ is an exact functor from
${\cal C}(\nabla)$ to $R-PC(\hat{\Delta})$.

Let $\Omega\subseteq\Gamma$ be a finite coideal. Denote $e_{\Omega}$ by $e$ and $eAe-Dis$ by ${\cal C}_{\Omega}$. By Theorem 3.5 (b) of \cite{cps}, $eAe=End_{PC-A}(eA)\simeq End_{\cal C}(\bigoplus_{\lambda\in\Omega} I(\lambda))^{\circ}$ is a finite-dimensional quasi-hereditary algebra with respect to the poset $\Omega$. 
Explicitly, the indecomposable projective summands of $eAe$ consists of all the $e_{\lambda}Ae$, where $top(e_{\lambda}Ae)=M(\lambda)e$, for all $\lambda\in\Omega$. 
If $\nu\not\in\Omega$, then $M(\nu)e=0$ and $\nabla(\nu)^* e=0$. It follows that any $e_{\lambda}Ae$ for $\lambda\in\Omega$  
has a finite standard filtration such that its top quotient is $\nabla(\lambda)^* e$. Dually, for every $\lambda\in\Omega$,   
$eI(\lambda)$ is an indecomposable injective  envelope of $eL(\lambda)$ in ${\cal C}_{\Omega}$ 
and $eI(\lambda)$ has a finite costandard filtration such that the bottom member is $e\nabla(\lambda)$.

For $\lambda\in\Omega$, define $P(\lambda)=\tau(I(\lambda))$.
Indecomposable projective modules in ${\cal C}_{\Omega}$ are given in the following lemma.
  
\begin{lm}\label{lm53} The module $eP(\lambda)$ is a projective cover of $eL(\lambda)$ in ${\cal C}_{\Omega}$.  
\end{lm}
\begin{proof}
By Remark \ref{rm42}, $P(\lambda)$ can be identified with $e_{\lambda}A$, where the left $A$-module action is given by 
$ax=x\phi(a)$ for  $x\in e_{\lambda}A$ and $a\in A$.
The map $e_{\lambda}a\mapsto \phi^{-1}(a)e_{\lambda}$ defines an isomorphism $P(\lambda)\simeq Ae_{\lambda}$ of left (discrete) $A$-modules.
Therefore $eP(\lambda)$ is an indecomposable projective cover of $e \, top(P(\lambda))=e L(\lambda)$. 
\end{proof}

\subsection{Extension functors}

Lemma \ref{lm53} implies that the modules $e\Delta(\lambda)$ for $\lambda\in\Omega$ are the corresponding standard objects in $C_{\Omega}$. 
Additionally, the modules $eT(\lambda)$ are the corresponding indecomposable tilting modules. 
For each $k\geq 1$ denote $\Gamma\setminus\Gamma_k$ by $\Omega_k$ and $e_{\Omega_k}$ by 
$e_k$.
Morphisms in the category ${\cal C}$ are recognized as inverse limits of their restrictions in the next lemma.

\begin{lm}\label{lm54}
For every $M, N\in {\cal C}$ there is a natural isomorphism $$Hom_{\cal C}(M, N)\simeq\lim\limits_{\leftarrow} 
Hom_{e_k Ae_k} (e_k M, e_k N).$$
\end{lm}
\begin{proof}
For every $m\in M$, there is an integer $k=k(m)\geq 0$ such that $e_{\Gamma_k} m=0$; consequently $e_k m=m$. Thus $M=\bigcup_{k\geq 0} e_k M$ and
$Hom_{\cal C}(M, N)\subseteq\lim\limits_{\leftarrow} Hom_{e_k Ae_k}(e_k M, e_k N)$. 

If $\lim\limits_{\leftarrow}\phi_k\in
\lim\limits_{\leftarrow} Hom_{e_k Ae_k}(e_k M, e_k N)$, then define $\phi(m)=\phi_t(m)$ for $m\in M$ and $t\geq k(m)$. For given $a\in A$ and $m\in M$, set $l=\max\{k(m), k(am), k(\phi(m)),  k(\phi(am))\}$. Then $$\phi(am)=\phi_l(am)=\phi_l(e_l ae_lm)=e_l a e_l\phi_l(m)=a \phi(m).$$
Hence $\phi\in Hom_{\cal C}(M,N)$.\end{proof}

For an injective resolution $I_N$ of $N$, there is a projective spectrum of complexes 
$\{Hom_{e_k Ae_k}(e_k M, e_k I_N)| k\geq 0\}$. Define a functor $N\to \lim\limits_{\leftarrow}H^*(Hom_{e_k Ae_k}(e_k M, e_k I_N))=
\lim\limits_{\leftarrow}Ext^*_{e_k Ae_k}(e_k M, e_k N)$. This functor is identified with the extension functor in ${\cal C}$ in the next lemma.
\begin{lm}\label{lm55}
The functor $N\to Ext^*_{\cal C}(M , N)$ is naturally isomorphic to $N\to \lim\limits_{\leftarrow} Ext^*_{e_k Ae_k}(e_k M , e_k N)$.
\end{lm}
\begin{proof}
Both functors are covariant $\delta$-functors erasable by injectives. The statement follows using Theorem 7.1 of Chapter XX in \cite{lang} and 
Lemma \ref{lm54}. 
\end{proof}

In the next step we prove a crucial proposition.
\begin{pr}\label{pr52}
For every $M\in {\cal C}(\nabla)$, there is a restricted tilting module $T$ and an epimorphism
$\pi : T\to M$ such that $\ker\pi\in {\cal C}(\nabla)$. 
\end{pr}
\begin{proof}
We shall modify the proof of Proposition A4.4 of \cite{don1}. For $k\geq 0$, set $M_k=O_{\Gamma_k}(M)$.  
By Lemma \ref{lm31}, each $M_k/M_{k+1}$ is isomorphic to a finite direct sum $\bigoplus_{\lambda\in\Gamma_k\setminus\Gamma_{k+1}}\nabla(\lambda)^{m_{\lambda}}$. 

We shall construct a direct spectrum of tilting modules and morphisms by considering, for every $k\geq 1$, the tilting module $T_k$ and the morphism $\pi_k: T_k \to M$ such that $T_k$ is a finite direct sum of $T(\lambda)$ for $\lambda\in\Omega_k$, and $\pi_k': T_k \to M \to M/M_k$  is an epimorphism whose kernel belongs to ${\cal C}(\nabla)$. Additionally, 
each $T_k$ shall be a direct summand of $T_{k+1}$ and $\pi_{k+1}|_{T_k}=\pi_k$. 

For each $k\geq 0$ there is an epimorphism $T_k'=
\bigoplus_{\lambda\in\Gamma_k\setminus\Gamma_{k+1}}T(\lambda)^{m_{\lambda}}\to M_k/M_{k+1}\to 0$ whose kernel belongs to ${\cal C}(\nabla)$.
By Theorem \ref{t36}, this epimorphism is induced by a morphism $\phi_k : T'_k\to M_k$. 
We define $T_1=T'_0$ and $\pi_1=\phi_0$. Proceeding by induction, we set 
$T_{k+1}=T_k\bigoplus T'_k$ and define $\pi_{k+1}: T_{k+1}\to M$ by 
$\pi_{k+1}(t+t')=\pi_k(t)+\phi(t')$ for $t\in T_k $ and $t'\in T'$. 

Define $T=\lim\limits_{\rightarrow} T_k$ and $\pi : T\to M$ such that $\pi=\lim\limits_{\rightarrow}\pi_k$. 
Then $T$ is clearly a restricted tilting module. Since
$e_k T=e_k T_k$ and $e_k T_k\to e_k M\simeq e_k(M/M_k)$ is an epimorphism for every $k\geq 1$, we conclude that $\pi$ is an epimorphism.
Denote $\ker\pi$ by $D$.
Then $Ext^1_{\cal C}(\Delta(\lambda), D)=\lim\limits_{\leftarrow}Ext^1_{e_k Ae_k}(e_k \Delta(\lambda), e_k D)=0$ for every $\lambda\in\Gamma$.
In fact, each $e_k D$ is isomorphic to the kernel of $e_k T\to e_k (M/M_k)$ which has a $\nabla$-filtration in ${\cal C}_{\Omega_k}$. 
Theorem \ref{t32} concludes the proof.
\end{proof}
Let $U, V\in R-PC$. Define $Ext_{R-PC}^*(U, V)$ as $H^*(Hom_{R-PC}(P_U, V))$, where $P_U$ is a projective resolution of $U$.
Corresponding to a basis $\{W\}$ of neighborhoods at zero consisting of open submodules in $V$, there is a projective spectrum of
complexes $\{Hom_{R-PC}(P_U, V/W) | W\}$.
Define a functor $U\to \lim\limits_{\leftarrow} H^*(Hom_{R-PC}(P_U, V/W))=\lim\limits_{\leftarrow} Ext_{R-PC}^*(U, V/W)$.
We shall show that this functor can be identified with the extension functor in $R-PC$.
\begin{lm}\label{lm56}
The functor $U\to Ext^*_{R-PC}(U , V)$ is naturally isomorphic to $U\to \lim\limits_{\leftarrow} Ext_{R-PC}^*(U, V/W)$.
\end{lm}
\begin{proof}
Both functors are contravariant $\delta$-functors coerasable by projectives. The claim follows from Theorem 7.1' of Chapter XX in 
\cite{lang}.
\end{proof}
For a finite coideal $\Omega$ of $\Gamma$, denote $\{\Delta(\lambda) |\lambda\in\Omega\}$ by $\Delta_{\Omega}$ and 
$e=e_{\Omega}$ as above.
The next lemma deals with an isomorphism of extension functors.
\begin{lm}\label{lm57} If $X$ has a finite $\Delta_{\Omega}$-filtration and $Y$ is restricted, then the natural map 
$Ext^i_{\cal C}(X, Y)\to Ext^i_{{\cal C}_{\Omega}}(eX, eY)$ is an isomorphism for every $i\geq 0$. 
\end{lm}
\begin{proof}
To modify the proof of Proposition A3.13 from \cite{don1} we need to make the following two observations.
First, if $\phi : \Delta(\lambda)\to Y$ is a non-zero morphism, then $\phi$ induces an isomorphism $e\Delta(\lambda)/eM\simeq e\phi(\Delta(\lambda))/e\phi(M)$, where $M=rad\Delta(\lambda)$. Secondly, $[Y : L(\lambda)]=[eY : eL(\lambda)]<\infty$ for every
$\lambda\in\Omega$. 
\end{proof}
To prepare the next statement, assume that $A$ is finite-dimensional, and therefore $\Gamma$ is finite. If $\Omega$ is a coideal of $\Gamma$, then
there are two Ringel functors 
$F : {\cal C}(\nabla)\to R-PC(\hat{\Delta})$ and $F' : {\cal C}_{\Omega}(\nabla)\to R'-PC(\hat{\Delta})$, where $R=End_{\cal C}(T)^{\circ}$,
$R'=End_{{\cal C}_{\Omega}}(eT)^{\circ}\simeq R/I$, and $I=Hom_{\cal C}(T/O^{\Gamma\setminus\Omega}(T), T)$.
In this setting we obtain the following commutative diagram.
\begin{lm}\label{lm58}
For every $M, N\in {\cal C}(\nabla)_f$ and $i\geq 0$ there is a commutative diagram,
$$\begin{array}{ccc}
Ext^i_{\cal C}(M, N) &\to & Ext^i_R(F(M), F(N)) \\
\downarrow & & \downarrow \\
Ext^i_{{\cal C}_{\Omega}}(eM, eN) &\to & Ext^i_{R'}(F'(eM), F'(eN))
\end{array},$$
where the horizontal rows are isomorphisms from Proposition A4.8(i) of \cite{don1}.
\end{lm} 
\begin{proof}
The same dimension shift argument used in Proposition A4.8(i) of \cite{don1} reduces the general case to the case $i=0$ which is obvious.
\end{proof}

The second main result of this section states that the functor $F$ preserves extensions of modules.

\begin{tr}\label{t52}
Let $M, N\in {\cal C}(\nabla)$. Then $Ext_{R-PC}^*(F(M), F(N))\simeq Ext^*_{\cal C}(M, N)$.
\end{tr}
\begin{proof}
By Proposition \ref{pr52}, there is a restricted tilting resolution, 
\[ \ldots\to T(i)\to\ldots\to T(1)\to M\to 0,\]
such that
$\ker(T(1)\to M)\in {\cal C}(\nabla)$ and $\ker(T(i)\to T(i-1))\in {\cal C}(\nabla)$ for $i\geq 2$. 
We shall denote this resolution by $T_M$.
Then $F(T_M)$ is a projective resolution of $F(M)$. 
Fix an integer $k\geq 0$. 
Then $R/I_k\simeq End_{{\cal C}_{\Omega_k}}(e_k T)^{\circ}=R_k$ by Lemma \ref{lm57}, and this isomorphism is compatible with 
$F(V)/F(V) I_k\simeq Hom_{\cal C}(O^{\Gamma_k}(T), V)\simeq Hom_{{\cal C}_{\Omega_k}}
(e_k T, e_k V)$ for $V\in {\cal C}(\nabla)$. In particular, $F(T_M)/F(T_M)I_k\simeq F_k(e_k T_M)$ is a projective resolution of $F(M)/F(M) I_k\simeq
F_k(e_k M)$ in $R_k-PC$. Lemma \ref{lm56} implies that $Ext^*_{R-PC}(F(M), F(N))$ is a projective limit of the spectrum
$\{Ext^*_{R_k -PC}(F_k(e_k M) , F_k(e_k N))| k\geq 1\}$. By Lemma \ref{lm58}, this spectrum is isomorphic to 
$\{Ext^*_{e_k Ae_k}(e_k M, e_k N)| k\geq 1\}$. Lemma \ref{lm55} concludes the proof.
\end{proof}

\begin{cor}\label{50}
There is an equivalence of categories of finite modules in ${\cal C}$ filtered by $\nabla$ and those finite modules filtered by standard modules over $R$.
\end{cor}

The following corollary generalizes the first part of Theorem \ref{t36} and Corollary \ref{c34}.
\begin{cor}\label{c51}
Let $\Gamma$ and $\widetilde{\Gamma}$ be finitely-generated ideals of $\Lambda$, $T$ be a $\Gamma$-restricted tilting object, and $M$ be a
$\widetilde{\Gamma}$-restricted object. If $M$ has an increasing (or decreasing) $\nabla$-filtration,
then $Ext^i_{\cal C}(T, M)=0$ for every $i\geq 1$.
\end{cor}
\begin{proof}
It is enough to observe that $F(T)$ is a projective pseudocompact $R$-module and use Theorem \ref{t52}.
\end{proof}

\section{Examples and concluding remarks}

The purpose of this section is to provide examples illustrating previously introduced concepts.

Let $G=GL(m|n)$ be a general linear supergroup. The category ${\cal C}=G-smod$ of (left) rational $G$-supermodules (with even morphisms) is equivalent to the category of (right) supercomodules over its coordinate Hopf superalgebra $K[G]$. 
For the definitions of $G$ and $K[G]$, we refer to \cite{bk,z}. In this section we shall only consider the example $G=GL(1|1)$.

\subsection{Category of $GL(1|1)$-supermodules} 

We start by explicitly describing the comultiplication map $\delta$ on $K[G]=K[c_{ij}| 1\leq i, j\leq 2]_{c_{11}c_{22}}$, where $|c_{ij}|=0$ if $i=j$ and $|c_{ij}|=1$ otherwise.
Define the left weight of a (rational) monomial $m=c_{11}^a c_{12}^b c_{21}^c c_{22}^d$, where $a, b, c, d\in\mathbb{Z}$ and $0\leq b, c\leq 1$, as $\lambda_l(m)=(a+b|c+d)$ (and similarly define $\lambda_r(m)=(a+c|b+d)$ for the right weight). The subspace of
$K[G]$ generated by all monomials with $\lambda_l(m)=\lambda$ shall be denoted by $_{\lambda}K[G]$.
Assume $\lambda=(i|j)$ and denote $|\lambda|=i+j$ by $r$. Then the superspace $_{\lambda}K[G]$ has a basis consisting of elements:
\[A_{\lambda}=c_{11}^i c_{22}^j, B_{\lambda}=c_{11}^{i-1} c_{12}c_{22}^j, C_{\lambda}=c_{11}^i c_{21}c_{22}^{j-1}, \text{ and } D_{\lambda}=c_{11}^{i-1}c_{12}c_{21}c_{22}^{j-1}.\]
We have 
\[\delta(A_{\lambda})=A_{\lambda}\otimes A_{\lambda} +iB_{\lambda}\otimes C_{\lambda-\pi}+
jC_{\lambda}\otimes B_{\lambda+\pi}+ijD_{\lambda}\otimes D_{\lambda},\]
\[\delta(B_{\lambda})=B_{\lambda}\otimes Y_{\lambda-\pi} +X_{\lambda}\otimes B_{\lambda},
\delta(C_{\lambda})=C_{\lambda}\otimes X_{\lambda+\pi}+Y_{\lambda}\otimes C_{\lambda}, \text{ and }\]
\[\delta(D_{\lambda})=D_{\lambda}\otimes A_{\lambda}-C_{\lambda}\otimes B_{\lambda+\pi}+B_{\lambda}\otimes C_{\lambda-\pi} +(A_{\lambda}+(j-i)D_{\lambda})\otimes D_{\lambda},\]
where $X_{\lambda}=A_{\lambda}+jD_{\lambda}, Y_{\lambda}=A_{\lambda}-iD_{\lambda}$ and $\pi=(1|-1)$. 
Since $K[G]$ is a direct sum of right $K[G]$-supercomodules $_{\lambda}K[G]$, each $_{\lambda}K[G]$ is an injective $G$-supermodule.

The category $\cal C$  is a highest weight category with respect to a poset $\Lambda$ that consists of weights $\lambda^{\epsilon}$, where $\lambda=(i|j)\in\mathbb{Z}^2$ and $\epsilon\in\{0, 1\}$.
The weights of $\Lambda$ are ordered by the dominant order $(i|j)^{\epsilon_1}\leq (k|l)^{\epsilon_2}$ if and only if $i\leq k$ and
$i+j=k+l$ (cf. Theorem 5.5 of \cite{z}). The poset $\Lambda$ is interval-finite and good; each weight $\lambda^{\epsilon}$ has exactly 
two predecessors, $(\lambda-\pi)^0$ and $(\lambda-\pi)^1$. 
For $V\in {\cal C}$, let $V^c$ denote {\it a parity shift} of $V$;
that is, $V^c_{\epsilon}=V_{\epsilon+1\pmod 2}$ for $\epsilon\in\{0, 1\}$. For instance, $L(\lambda^0)^c=L(\lambda^1)$. The costandard and standard objects in the category $\cal C$ are the induced supermodules $H^0(\lambda^{\epsilon})$ and the Weyl supermodules $V(\lambda^{\epsilon})$, respectively. In the notations of \cite{z}, we have $H^0(\lambda^0)=H^0(\lambda), V(\lambda^0)=V(\lambda)$, $H^0(\lambda^1)=H^0(\lambda)^c$ and  $V(\lambda^1)=V(\lambda)^c$.    

The following lemma describes explicitly the simple, costandard and injective $G$-supermodules corresponding the weight $\lambda^0$.
Analogous statements are valid for the weight $\lambda^1$.

\begin{lm}\label{lm61} a) The supermodule $\nabla(\lambda^0)$ is isomorphic to the supermodule $K$-spanned by $B_{\lambda}$ and $X_{\lambda}$.

b) If $p$ divides $r$, then $L(\lambda^0)=KX_{\lambda}$; otherwise $L(\lambda^0)=\nabla(\lambda^0)(=\Delta(\lambda^0))$.

c) If $p$ divides $r$, then $_{\lambda}K[G]=I(\lambda^0)$; otherwise $_{\lambda}K[G]=\nabla(\lambda^0)\bigoplus$ $\nabla((\lambda+\pi)^1)$ and $I(\lambda^0)=\nabla(\lambda^0)$. In both cases $I(\lambda^0)$ is a tilting module.\end{lm}
\begin{proof} 
a) Denote $V=KC_{\lambda}+KY_{\lambda}$ and $W=KB_{\lambda}+KX_{\lambda}$. Using the above formulas for $\delta$, we obtain $\delta(X_{\lambda})=X_{\lambda}\otimes X_{\lambda}+rB_{\lambda}\otimes C_{\lambda-\pi}$ and $\delta(Y_{\lambda})=Y_{\lambda}\otimes Y_{\lambda}+rC_{\lambda}\otimes B_{\lambda+\pi}$. Therefore, $V$ and $W$ are subsupermodules of $_{\lambda}K[G]$. 
Elementary computation shows that $W=O_{(\lambda^0]}(_{\lambda}K[G])$. Thus $W$ has a $\nabla$-filtration; and since $B_{\lambda}$ generates $W$,  $W=\nabla(\lambda^0)$.

b) This follows immediately from a). 

c) We compute $_{\lambda}K[G]/\nabla(\lambda^0)\simeq\nabla((\lambda+\pi)^1)$. 
If $p$ does not divide $r$, then $V\bigcap W=0$. This implies $_{\lambda}K[G]=V\bigoplus W$ and
$V=\nabla((\lambda+\pi)^1)=\Delta((\lambda+\pi)^1)$. 
If $p|r$, then $X_{\lambda}=Y_{\lambda}$ and $V\bigcap W$ is the socle of $_{\lambda}K[G]$. 
In this case, the supermodule
$_{\lambda}K[G]$ has a composition series:
$$\begin{array}{ccccc}&&L(\lambda^0)&&\\&\diagup&&\diagdown&\\
L((\lambda-\pi)^1)&&&&L((\lambda+\pi)^1)\\
&\diagdown&&\diagup&\\
&&L(\lambda^0)&&
\end{array}.$$
In particular, $_{\lambda}K[G]=I(\lambda^0)$ has a $\Delta$-filtration with quotients $\Delta((\lambda+\pi)^1)$ and $\Delta(\lambda^0)$. Thus
$I(\lambda^0)$ is a tilting module.
\end{proof}
\begin{rem}\label{rm61}
Each $I(\lambda^{\epsilon})$ is selfinjective and Chevalley dual to itself. In particular, it is a projective cover of
$L(\lambda^{\epsilon})$.
\end{rem}

Fix $r\in \mathbb{Z}$, consider an ideal $\Gamma_r=((r|0)^0]\bigcup ((r|0)^1]$, and denote ${\cal C}_r={\cal C}[\Gamma_r]$.
To simplify notations denote $(r-i|i)^{\epsilon}$ by $i^{\epsilon}$ for each $i\geq 0$.
In particular, $i^{\epsilon_1} > j^{\epsilon_2}$ if and only if $i< j$.
Denote $O_{\Gamma_r}(K[G])=\nabla(0^0)\bigoplus_{i > 0} \phantom{i}_{i}K[G]$ by $C_r$. Then $C_r$ is a subsupercoalgebra of $K[G]$ and the category ${\cal C}_r$ coincides with the category 
of (right) $C_r$-supercomodules (cf. \cite{don2, green, z}). 
In fact, $M\in {\cal C}_r$ if and only if $M$ is embedded in $O_{\Gamma_r}(\bigoplus_{i^{\epsilon}\in\Gamma_r}
I(i^{\epsilon})^{m_{i^{\epsilon}}})$ which is true if and only if $M$ is embedded in $C_r^m\bigoplus (C^c_r)^n$ (where $m$ and $n$ are possibly infinite).  
This follows from $cf(C_r)=cf(C_r^c)=C_r$. 

As above, ${\cal C}_r$ is equivalent to the category of (left) discrete $S_r$-supermodules
over the pseudocompact superalgebra $S_r=C^{*}_{r}$. 
Recall that if $\delta(c)=\sum c_1\otimes c_2$ and $\tau_M(m)=\sum m_1\otimes c_2$ for $c\in C_r$, $M\in {\cal C}_r$ and $m\in M$, then 
$(xy)(c)=\sum(-1)^{|y||c_1|}x(c_1)y(c_2)$ and $xm=\sum(-1)^{|m_1||x|}x(c_2)m_1$ for $x, y\in S_r$.   
We shall call $S_r$ a {\it pseudocompact Schur superalgebra} corresponding to the ideal $\Gamma_r$. 

\subsection{Pseudocompact Schur algebra $S_r$}

The following lemma describes multiplication of generating elements in the dual of a supercoalgebra $C$.
\begin{lm}\label{lm62}
Let $C$ be a supercoalgebra; and let $\{c_i\}_{i\in I}$ be a homogeneous basis of $C$ and $\delta_C(c_i)=\sum_{k, l\in I}\alpha_{i, kl} c_k\otimes c_l$ 
for $i\in I$. Then $c^*_k c^{*}_{l}=(-1)^{|c_k||c_l|}\sum_{i\in I}\alpha_{i, kl}c^{*}_{i}$ for every $k, l\in I$.
\end{lm}
\begin{proof}
This follows from straightforward computation.
\end{proof}

The structure of $S_r$ is determined in the following lemma.

\begin{lm}\label{lm625}
If $p$ does not divide $r$, then $S_r$ is a product of matrix superalgebras $M(1|1)$.
If $p$ divides $r$, then $S_r$ is a product of indecomposable projective factors $S_r(i)$ for $i\geq 0$. Besides, for $i>0$, each $S_r(i)$ has a composition series:
$$\begin{array}{ccccc}&&M(i^0)&&\\
&\diagup&&\diagdown&\\
M((i-1)^{1})&&&&M((i+1)^{1})\\
&\diagdown&&\diagup&\\
&&M(i^0)&&
\end{array};$$
and $S_r(0)$ has a composition series:
$$\begin{array}{c}
M(0^0) \\
| \\
M(1^1)
\end{array}.
$$
\end{lm}
\begin{proof}
Assume that $p$ does not divide $r$. Then ${\cal C}_r$ is a semisimple category and $\Delta(i^{\epsilon})=\nabla(i^{\epsilon})=L(i^{\epsilon})$. In this case, $C_r=
\bigoplus_{i\geq 0}(L(i^0)\bigoplus L(i^1))$, where each $C_r(i)=L(i^0)\bigoplus L(i^1)$ is a subsupercoalgebra
generated by $X_i, B_i, Y_{i+1}$ and $C_{i+1}$. 
Lemma \ref{lm62} implies that $S_r\simeq\prod_{i\geq 0}(C_r(i))^*$, where each $(C_r(i))^*\simeq M(1|1)$ is a matrix superalgebra. 
The isomorphism is given by $X_i^*\mapsto e_{11}, B_i^*\mapsto (-r)^{\frac{1}{2}}e_{12}, Y_{i+1}^*\mapsto e_{22}$ and $C_{i+1}^*\mapsto (-r)^{\frac{1}{2}}e_{21}$.
Moreover,
$V\in {\cal C}_r(\nabla)={\cal C}_r$ if and only if $V=\bigoplus_{i\geq 0} V(i)$, where $S_r(j) V_i=\delta_{ij}V_i$ for $i, j\geq 0$ and every summand $V_i$ is a finite direct sum of simple supermodules isomorphic to the standard $M(1|1)$-supermodule $K^{1|1}$ or to its parity shift.

If $p$ divides $r$, then $S_r$ is a product of its indecomposable projective factors: $S_r (0)=\nabla(0^0)^*$ and $S_r (i)=\phantom{i}_iK[G]^*$ for $i>0$. In this case $C^*_r$ is a $K$-span of elements $X_i^* =Y_i^*$ and $B_i^*$ for $i\geq 0$, and 
$C_i^*$ and $D_i^*$ for $i>0$.
By Lemma \ref{lm62}, the multiplication in
$C^{*}_{r}$ is given by 
\[
\begin{array}{|c|c|c|c|c|} \hline
      & X^{*}_{j} & B^{*}_{j} & C^{*}_{j} & D^{*}_{j} \\ \hline 
X^{*}_{i} & \delta_{ij}X^{*}_{j} & \delta_{ij}B^{*}_{j} & \delta_{ij}C_j^* & \delta_{ij}D^{*}_{j} \\ \hline 
B^{*}_{i} & \delta_{i,j-1}B^{*}_{j-1} & 0 & -(1-\delta_{1,j})\delta_{i,j-1}D^{*}_{j-1} & 0 \\ \hline
C^{*}_{i} & \delta_{i,j+1}C^{*}_{j+1} & \delta_{i,j+1}D^{*}_{j+1} & 0 & 0 \\ \hline
D^{*}_{i} & \delta_{ij}D^{*}_{j} & 0 & 0 & 0 \\ \hline
\end{array}.
\]
Thus $S_r(0)=X^{*}_{0} S_r$ and $S_r(i) =X^{*}_{i} S_r$ for $i>0$. 
The claim follows.
\end{proof}

More explicitly, the algebra $S_r$ can be described by the quiver 
\[0^0\overset{\underrightarrow{\beta_1}}{_{\underleftarrow{\gamma_0}}}
1^1 \overset{\underrightarrow{\beta_2}}{_{\underleftarrow{\gamma_1}}} 2^0 \ldots (i-1)^{(-1)^{i-1}} \overset{\underrightarrow{\beta_{i}}}{_{\underleftarrow{\gamma_{i-1}}}}
i^{(-1)^i} \overset{\underrightarrow{\beta_{i+1}}}{_{\underleftarrow{\gamma_{i}}}} 
(i+1)^{(-1)^{i+1}} \overset{\underrightarrow{\beta_{i+2}}}{_{\underleftarrow{\gamma_{i+1}}}}
\ldots, 
\]
and the relations $\beta_{i+1}\beta_i=\gamma_{i-1}\gamma_i=0$ and $\gamma_{i}\beta_{i+1}=-\beta_{i}\gamma_{i-1}$ for every $i>0$, and 
$\gamma_0\beta_1=0$.

The algebra $S_r$ is an ascending quasi-hereditary pseudocompact algebra with a defining system of ascending ideals, each denoted by some $H_i$ for $i\geq 0$, where
each $H_i$ is a $K$-span of elements $X^*_j, B^*_j$ for $0\leq j\leq i$ and $C^*_j, D^*_j$ for $1\leq j\leq i+1$.

\begin{rem}\label{r61}

We would like to compare the structure of the above algebra $S_r$ to the structure of the Schur superalgebra $S_r(1|1)$ described  
in \cite{mz}. The latter is a finite-dimensional algebra that corresponds to polynomial representations of $GL(1|1)$ of degree $r$; and in the case when $p$ divides $r$ it is not quasi-hereditary, whereas our above algebra $S_r$ is an ascending quasi-hereditary pseudocompact algebra. 
Hence it appears that the framework of pseudocompact algebras is more suitable in this setting.
\end{rem}

Every superalgebra $A$ is a $\mathbb{Z}_2$-module, where the generator of $\mathbb{Z}_2$ acts on $A$ as $a\mapsto (-1)^{|a|}a$.
Denote the semi-direct product algebra $A \rtimes \mathbb{Z}_2$ by $\Hat{A}$. If $A$ is pseudocompact, then so is $\Hat{A}$. 
We could switch from the category of $A$-supermodules to that of modules over $\Hat{A}$ using the following lemma.

\begin{lm}\label{lm63}
The category $\Hat{A}-Dis$ is naturally identified with the category of (left) discrete $A$-supermodules.
\end{lm}
\begin{proof}
$\Hat{A}$ has a $\mathbb{Z}_2$-grading $A_0\bigoplus A_1$, where each $A_{\epsilon}$ coincides with $A$ as a vector space.
If we write each $\hat{a}\in \Hat{A}$ as $\hat{a}=a_0+a_1$, where $a_0\in A_0$ and $a_1\in A_1$, then $a_{\epsilon}b_{\mu}=(-1)^{|b|\epsilon}(ab)_{\epsilon+\mu\pmod 2}$.
If $V$ is a discrete $A$-supermodule, then $\Hat{A}$ acts on $V$ via $a_{\epsilon}v=(-1)^{|v|\epsilon}av$. Conversely, if $W\in \Hat{A}-Dis$,
then $W$ has an $A$-supermodule structure given by $W_{\epsilon}=\{w\in W| (1_A)_1 w=(-1)^{\epsilon}w\}$.  
\end{proof}

\subsection{Ringel dual of $\Hat{S_r}$}

Denote by $R_r$ the Ringel dual of $\Hat{S_r}$.
The structure of $R_r$ is given in the following lemma.

\begin{lm}\label{lm64}
If $p$ does not divide $r$, then $R_r$ is isomorphic to a product $\prod_{i^{\epsilon}\in\Gamma_r} \Hat{M}(i^{\epsilon})$, where 
each $\Hat{M}(i^{\epsilon})\simeq K$.
If $p$ divides $r$, then $R_r$ is a product of $R^0_r$ and $R_r^1$, where $R_r^{\epsilon}=\prod_{i\geq 0}\hat{P}(i^{i+\epsilon\pmod 2})$ for $\epsilon=0, 1$.
Here, for $i>0$, each $\hat{P}(i^{\epsilon})$ has a composition series:  
$$\begin{array}{ccccc}&&\hat{M}(i^{\epsilon})&&\\&\diagup&&\diagdown&\\
\hat{M}((i+1)^{\epsilon'})&&&&\hat{M}((i-1)^{\epsilon'})\\
&\diagdown&&\diagup&\\
&&\hat{M}(i^{\epsilon})&&
\end{array};$$
and $\hat{P}(0^{\epsilon})$ has a composition series:
$$\begin{array}{c}
\hat{M}(0^{\epsilon}) \\
| \\
\hat{M}(1^{\epsilon'}) \\
| \\
\hat{M}(0^{\epsilon})
\end{array}.
$$
\end{lm}

\begin{proof}
If $p$ does not divide $r$, then $R_r\simeq K^{|\Gamma_r|}=\prod_{i^{\epsilon}\in\Gamma_r}Ke_{i^{\epsilon}}$, where $Ke_{i^{\epsilon}}=\Hat{\Delta}(i^{\epsilon})=
\Hat{M}(i^{\epsilon})$. Thus $R_r-PC=R_r-PC(\Hat{\Delta})$; and any pseudocompact $R_r$-supermodule
$V$ is isomorphic to $\prod_{i^{\epsilon}\in\Gamma_r}V_{i^{\epsilon}}$, where each $V_{i^{\epsilon}}=
e_{i^{\epsilon}}V$ is finite-dimensional. In particular, $F$ induces an equivalence ${\cal C}_r\simeq R_r-PC$. 

Next, assume that $p|r$. For simplicity, denote $\epsilon +1\pmod 2$ by $\epsilon'$. By the above,
$T(i^{\epsilon})=I((i+1)^{\epsilon'})$. By parts (4) and (5) of Corollary \ref{c32}, the projective factor  
$$\hat{P}(i^{\epsilon})=Hom_{{\cal C}_r}(I((i+2)^{\epsilon})\bigoplus I((i+1)^{\epsilon'})\bigoplus
I(i^{\epsilon})  , I((i+1)^{\epsilon'}))$$
is four-dimensional for $i > 0$, and 
$$\hat{P}(0^{\epsilon})=Hom_{{\cal C}_r}(I(2^{\epsilon})\bigoplus I(1^{\epsilon'})  , I(1^{\epsilon'}))$$
is three-dimensional. 
If $i> 0$, then the elements $e_{i^{\epsilon}}$, 
$$b_{(i-1)^{\epsilon}} : I(i^{\epsilon})\to \Delta(i^{\epsilon})\to I((i+1)^{\epsilon'}),$$
$$c_{(i+1)^{\epsilon}} : I((i+2)^{\epsilon})\to \nabla((i+1)^{\epsilon'})\to I((i+1)^{\epsilon'}), \text{ and}$$ 
$$d_{i^{\epsilon}} : I((i+1)^{\epsilon'})\to L((i+1)^{\epsilon'})\to I((i+1)^{\epsilon'})$$
form a basis of $\hat{P}(i^{\epsilon})$.
The module $\hat{P}(0^{\epsilon})$ has a basis consisting of elements $e_{0^{\epsilon}}, d_{0^{\epsilon}}$, and
$c_{1^{\epsilon}}$. The multiplication table of basis elements is given as
$$
\begin{array}{|c|c|c|c|c|} \hline
      & e_{j^{\epsilon}} & b_{j^{\epsilon}} & c_{j^{\epsilon}} & d_{j^{\epsilon}} \\ \hline 
e_{i^{\mu}} & \delta_{ij}\delta_{\mu\epsilon}e_{j^{\epsilon}} & \delta_{i,j}\delta_{\mu'\epsilon}b_{j^{\epsilon}} & \delta_{i,j}\delta_{\mu'\epsilon}c_{j^{\epsilon}} & \delta_{ij}\delta_{\mu\epsilon}d_{j^{\epsilon}} \\ \hline 
b_{i^{\mu}} & \delta_{i,j-1}\delta_{\mu\epsilon}b_{(j-1)^{\epsilon}} & 0 & \delta_{i, j-1}\delta_{\mu'\epsilon}d_{(j-1)^{\epsilon}} & 0 \\ \hline
c_{i^{\mu}} & \delta_{i,j+1}\delta_{\mu\epsilon}c_{(j+1)^{\epsilon}} & \delta_{i, j+1}\delta_{\mu'\epsilon}d_{(j+1)^{\epsilon}} & 0 & 0 \\ \hline
d_{i^{\mu}} & \delta_{ij}\delta_{\mu\epsilon}d_{j^{\epsilon}} & 0 & 0 & 0 \\ \hline
\end{array}.
$$
The composition series for $\hat{P}(i^{\epsilon})$ is then clear.
The algebra $R_r$ is a product of $R^0_r$ and $R_r^1$, where $R_r^{\epsilon}=\prod_{i\geq 0} \hat{P}(i^{i+\epsilon\pmod 2})$ for $\epsilon=0, 1$.
\end{proof}

The algebra $R_r$ is a descending quasi-hereditary pseudocompact algebra with a defining system of descending ideals, each denoted by some $G_i$ for $i\geq 0$, where
$G_i$ is generated topologically by elements $e_{j^{\epsilon}}$ and $c_{j^{\epsilon}}$ for $j\geq i$ and 
$b_{j^{\epsilon}}$ and $d_{j^{\epsilon}}$ for $j\geq i-1$ for $\epsilon=0,1$. Thus elements of $G_i$ are infinite sums of these generators multiplied by coefficients from $K$.

Comparing the corresponding multiplication tables, we see that $S_r$ is isomorphic to factoralgebras $R_r^0/\langle d_{0^0} \rangle$ and  $R_r^1/\langle d_{0^1} \rangle$ 
via the maps  
$e_{i^{\epsilon}}\leftrightarrow X^{*}_{i}, (-1)^i b_{i^{\epsilon'}}\leftrightarrow B^{*}_{i}, c_{i^{\epsilon'}}\leftrightarrow C^{*}_{i}$ and 
$(-1)^{i+1} d_{i^{\epsilon}}\leftrightarrow D^{*}_{i}$ for $\epsilon=0,1$.

\subsection{Topics for further investigation}
We conclude with some open problems.

It is natural to expect that $F$ induces an equivalence ${\cal C}(\nabla)\simeq
R-PC(\Hat{\Delta})$.

For any $m, n\geq 1,$ the category $GL(m|n)-smod$ of supermodules over the general linear supergroup $GL(m|n)$ is a highest weight category with respect to a poset $\Lambda=\{\lambda^{\epsilon}=(\lambda_1, \ldots, \lambda_{m+n})^{\epsilon}|
\lambda_1\geq\ldots\geq\lambda_m ; \lambda_{m+1}\geq\ldots\geq\lambda_{m+n}; \epsilon=0, 1\}$. (See \cite{z}.)
For a (finitely-generated) ideal $\Gamma\subseteq
\Lambda$, one can define a pseudocompact Schur superalgebra $S_{\Gamma}$ analogously as before.

We conclude by proposing the following topics for further investigations.

1) Describe Schur superalgebras $S_{\Gamma}$ for all $\Gamma$.

2) Describe tilting objects in $(GL(m|n)-smod)[\Gamma]$, and determine if they are finite-dimensional.

3) Describe the Ringel dual of $S_{\Gamma}$.

4) Let $OSp(m|2n)$ be an ortho-symplectic supergroup. Determine whether the category of rational $OSp(m|2n)$-supermodules is a highest weight category.


\begin{thebibliography}{99}
\bibitem{un} Anderson, F. W. and Fuller, K. R.: {\em Rings and categories of modules}, 
Second edition. Graduate Texts in Mathematics, {\bf 13}. Springer-Verlag, New York, 1992.
\bibitem{ars} Auslander, M., Reiten, I. and Smalø, S.: {\em Representation theory of Artin algebras}, 
Cambridge Studies in Advanced Mathematics, {\bf 36}. Cambridge University Press, Cambridge, 1997.
\bibitem{bb} Brenner, S., Butler, M. C. R.: {\em Generalizations of the Bernstein-Gel'fand-Ponomarev reflection functors.} 
Representation theory, II (Proc. Second Internat. Conf., Carleton Univ., Ottawa, Ont., 1979), 103–-169, Lecture Notes in Math., 832, Springer, Berlin-New York, 1980.
\bibitem{br} Brumer, A.: {\em Pseudocompact algebras, profinite groups and class formations},
J. Algebra {\bf 4} (1966), 442--470. 
\bibitem{bk} Brundan, J. and Kujawa, J.: {\em A new proof of the
Mullineux conjecture}, J.Algebraic.Combin. {\bf 18} (2003), 13-39.
\bibitem{bd} Bucur, I. and Deleanu, A.: {\em Introduction to the theory of categories and functors},
Pure and Applied Mathematics, Vol. XIX, Interscience Publication John Wiley \& Sons, Ltd., London-New York-Sydney, 1968. 
\bibitem{cps} Cline, E., Parshall, B. and  Scott, L.: {\em Finite-dimensional algebras and highest weight categories},
J. Reine Angew. Math. {\bf 391} (1988), 85--99. 
\bibitem{dr}  Dlab, V. and Ringel, C. M.: {\em The module theoretical approach to quasi-hereditary algebras},
Representations of algebras and related topics (Kyoto, 1990), 200--224, London Math. Soc. Lecture Note Ser., {\bf 168}, Cambridge Univ. Press, Cambridge, 1992. 
\bibitem{don0} Donkin, S.: {\em A filtration for rational modules.}, Math. Z. {\bf 177} (1981), no. 1, 1–8.
\bibitem{don2} Donkin, S.: {\em On Schur algebras and related algebras, I}, J. Algebra {\bf 104} (1986), no. 2, 310--328.
\bibitem{don3} Donkin, S.: {\em On projective modules for algebraic groups.}, J. London Math.Soc. (2), {\bf 54} (1996), no.1, 75-88.
\bibitem{don1}  Donkin, S.: {\em The $q$-Schur algebra}, London Mathematical Society Lecture Note Series, {\bf 253}. Cambridge University Press, Cambridge, 1998.
\bibitem{gab} Gabriel, P.: {\em Des catégories abéliennes. (French)}, Bull. Soc. Math. France {\bf 90} (1962), 323--448. 
\bibitem{green} Green, J.A.: {\em Locally finite representations}, J.Algebra, {\bf 41} (1976), 137-171.
\bibitem{jan} Jantzen, J. C.: {\em Representations of algebraic groups}, Pure and Applied Mathematics, {\bf 131}. Academic Press, Inc., Boston, MA, 1987.
\bibitem{lang} Lang, S.: {\em Algebra}, third revised edition, Springer, 2002.
\bibitem{mz} Marko, F. and Zubkov, A.N.: {\em Schur superalgebras in characteristic $p$}, Algebr. Represent. Theory {\bf 9} (2006), no. 1, 1–12. 
\bibitem{mt} Mitchell, B.: {\em Theory of categories}, Pure and Applied Mathematics, Vol. XVII Academic Press, New York-London, 1965.
\bibitem{sim} Simson, D.: {\em Coalgebras, comodules, pseudocompact algebras and tame comodule type}, Colloq. Math. {\bf 90} (2001), no. 1, 101--150.
\bibitem{tak}  Takeuchi, M.: {\em Morita theorems for categories of comodules}, J. Fac. Sci. Univ. Tokyo Sect. IA Math. {\bf 24} (1977), no. 3, 629--644. 
\bibitem{v} Van den Bergh, M.: {\em Blowing up of non-commutative smooth surfaces}, Mem. Amer. Math. Soc. {\bf 154} (2001), no. 734.
\bibitem{vv} Van Gastel, M. and Van den Bergh, M.: {\em Graded modules of Gelfand-Kirillov dimension 
one over three-dimensional Artin-Schelter regular algebras}, J. Algebra {\bf 196} (1997), no. 1, 251--282.
\bibitem{z} Zubkov, A.N: {\em Some properties of general linear supergroups and
Schur superalgebras}, Algebra and Logic (Algebra i Logika,
Russian) {\bf 45} (2006), 257-299.

\end{thebibliography}
\end{document}